\documentclass[12pt,oneside]{article}
\usepackage[inner=32mm,outer=31mm,tmargin=30mm,bmargin=40mm]{geometry}
\usepackage{hyperref}
\usepackage{sectsty}
    \sectionfont{\large}
    \subsectionfont{\normalsize}
\usepackage[utf8]{inputenc}
\usepackage[T1]{fontenc}
\usepackage{amssymb}
\usepackage{geometry}
\usepackage{amsmath}
\usepackage{amsthm}
\usepackage{nicefrac}
\usepackage{xfrac}
\usepackage{graphicx}
\usepackage{color}
\usepackage[normalem]{ulem}
\usepackage[english]{babel}
\usepackage{array}
\usepackage[onehalfspacing]{setspace}
\usepackage{verbatim}

\usepackage{rotating}

\numberwithin{equation}{section}

\newcommand\Item[1][]{%
  \ifx\relax#1\relax  \item \else \item[#1] \fi
  \abovedisplayskip=0pt\abovedisplayshortskip=0pt~\vspace*{-\baselineskip}}

\newcommand{\IR}{\ensuremath{\mathbb{R}}}
\newcommand{\IN}{\ensuremath{\mathbb{N}}}
\newcommand{\IZ}{\ensuremath{\mathbb{Z}}}

\newcommand{\IC}{\ensuremath{\mathbb{C}}}
\newcommand{\IK}{\ensuremath{\mathbb{K}}}

\newcommand{\norm}[1]{\left\Vert#1\right\Vert}

\newcommand{\set}[1]{\left\{#1\right\}}
\newcommand{\abs}[1]{\left|#1\right|}
\newcommand{\brackets}[1]{\left(#1\right)}
\renewcommand{\d}{{\rm d}} 
\newcommand{\scalar}[2]{\left\langle#1,#2\right\rangle}

\DeclareMathOperator{\vspan}{span}
\newcommand{\diff}{D}

\DeclareMathOperator{\rank}{rank}

\newtheorem{thm}{Theorem}

\newtheorem{cor}{Corollary}
\theoremstyle{plain}
\newtheorem{lemma}{Lemma}

\theoremstyle{definition}
\newtheorem{ex}{Example}

\title{Tensor power sequences and the approximation of tensor product operators} 

\author{David Krieg (david.krieg@uni-jena.de)}

\date{\today}

\begin{document}

\maketitle

\begin{abstract}
The approximation numbers of the $L_2$-embedding of
mixed order Sobolev functions on the $d$-torus are well studied.
They are given as the nonincreasing rearrangement of the $d$th tensor power of
the approximation number sequence in the univariate case.
I present results on the asymptotic and preasymptotic behavior for tensor powers
of arbitrary sequences of polynomial decay.
This can be used to study the approximation numbers of many other tensor product operators,
like the embedding of mixed order Sobolev functions on the $d$-cube into $L_2\brackets{[0,1]^d}$
or the embedding of mixed order Jacobi functions on the $d$-cube into $L_2\brackets{[0,1]^d,w_d}$ with Jacobi weight $w_d$.
\end{abstract}

\section{Introduction and Results}

Let $\sigma: \IN \to \IR$ be a nonincreasing zero sequence.
For any natural number $d$, its $d$th tensor power is the sequence $\sigma_d: \IN^d \to \IR$,
where
\begin{equation}
 \sigma_d(n_1,\dots,n_d)=\prod_{j=1}^d\sigma(n_j).
\end{equation}
Any such sequence $\sigma_d$ can then be uniquely rearranged to a nonincreasing zero sequence $\tau:\IN\to\IR$.
Tensor power sequences like this occur naturally in the study of approximation numbers of tensor power operators.
If $\sigma$ is the sequence of approximation numbers of a compact operator between two Hilbert spaces,
then $\tau$ is the sequence of approximation numbers of the compact $d$th tensor power operator
between the tensor power spaces.
%For example, the embeddings of $H^{s_1}_{\rm mix}\brackets{G^d}$ into $H^{s_2}_{\rm mix}\brackets{G^d}$
%for $s_1>s_2$ are of this type.
%Here, $G$ is a compact manifold and $H^{s}_{\rm mix}\brackets{G^d}$ is the Hilbert space of real or complex-valued functions on $G^d$
%with dominating mixed smoothness $s\in \IR$ and integrability 2.

What can we say about the behavior of $\tau$
based on the behavior of $\sigma$?
A classical result of Babenko \cite{babenko} and Mityagin \cite{mityagin} is concerned with
the speed of decay of these sequences:

\begin{thm}
\label{babenko mityagin theorem}
 Let $\sigma$ be a nonincreasing zero sequence and 
 $\tau$ be the nonincreasing rearrangement of its $d$th tensor power.
 For any $s>0$, the following holds.
 \begin{itemize}
  \item[(i)] If $\sigma(n) \preccurlyeq n^{-s}$, then $\tau(n) \preccurlyeq  n^{-s}\brackets{\log n}^{s(d-1)}$.
  \item[(ii)] If $\sigma(n) \succcurlyeq n^{-s}$, then $\tau(n) \succcurlyeq  n^{-s}\brackets{\log n}^{s(d-1)}$. 
 \end{itemize}
\end{thm}

Here, the symbol $\preccurlyeq$ (respectively $\succcurlyeq$) means that the
left (right) hand side is bounded above by a constant multiple
of the right (left) hand side for all $n\in\IN$.
Of course, other decay assumptions on $\sigma$ may be of interest.
For instance, Pietsch \cite{pietsch} and König \cite{koenig} study the decay of $\tau$,
if $\sigma$ lies in the Lorentz sequence space $\ell_{p,q}$ for positive indices $p$ and $q$,
which is a stronger assumption than $(i)$ for $s=1/p$ but weaker than $(i)$ for any $s>1/p$.
However, since we are motivated by the example of Sobolev embeddings,
we will stick to the assumptions of Theorem~\ref{babenko mityagin theorem}.
One of the problems with this theorem is that
it does not provide explicit estimates for $\tau(n)$, even if $n$ is huge.
This is because of the constants hidden in the notation.
But Theorem~\ref{babenko mityagin theorem} can be sharpened.

\begin{thm}
\label{asymptotic theorem}
 Let $\sigma$ be a nonincreasing zero sequence and
 $\tau$ be the nonincreasing rearrangement of its $d$th tensor power.
 For $c>0$ and $s>0$, the following holds.
 \begin{itemize}
  \item[(i)] If $\sigma(n) \lesssim c\, n^{-s}$, then $\tau(n) \lesssim \frac{c^d}{{(d-1)!}^s}\, n^{-s}\brackets{\log n}^{s(d-1)}$.
  \item[(ii)] If $\sigma(n) \gtrsim c\, n^{-s}$, then $\tau(n) \gtrsim \frac{c^d}{{(d-1)!}^s}\, n^{-s}\brackets{\log n}^{s(d-1)}$. 
 \end{itemize}
\end{thm}
% Der Fall c=infty bzw c=0 ist je entweder trivial oder folgt aus der Aussage.

We write $f(n) \lesssim g(n)$ for positive sequences $f$ and $g$
and say that $f(n)$ is asymptotically smaller or equal than $g(n)$,
if the limit superior of $f(n)/g(n)$ is at most one as $n$ tends to infinity.
Analogously, $f(n)$ is asymptotically greater than or equal to $g(n)$, write $f(n)\gtrsim g(n)$,
if the limit inferior of this ratio is at least one.
Finally, we say $f(n)$ is asymptotically equal to $g(n)$ and write $f(n)\simeq g(n)$ if the limit of the ratio equals one.
In particular, we obtain that $\sigma(n) \simeq c\, n^{-s}$ implies that
$\tau(n) \simeq \frac{c^d}{{(d-1)!}^s}\, n^{-s}\brackets{\log n}^{s(d-1)}$.
Theorem~\ref{asymptotic theorem} is due to Theorem~4.3 in \cite{ksu}. There, Kühn, Sickel and Ullrich prove
this asymptotic equality in an interesting special case:
$\tau$ is the sequence of approximation numbers for the $L_2$-embedding of the tensor power space
$H^{s}_{\rm mix}\brackets{\mathbb{T}^d}$ on the $d$-torus $[0,2\pi]^d$,
equipped with a tensor product norm.
The statement can be deduced from this special case with the help of their Lemma~4.14.
However, we prefer to give a direct proof in Section~\ref{asymptoticssection} by generalizing
the proof of Theorem~4.3 in \cite{ksu}.

Theorem \ref{asymptotic theorem} gives us a pretty good understanding of the asymptotic behavior of the $d$th tensor power
$\tau$ of a sequence $\sigma$ of polynomial decay.
If $\sigma(n)$ is roughly $c\, n^{-s}$ for large $n$,
then $\tau(n)$ is roughly $c^d\brackets{\frac{\brackets{\log n}^{d-1}}{\brackets{d-1}!}}^s n^{-s}$
for $n$ larger than a certain threshold.
But even for modest values of $d$, the size of this threshold may
go far beyond the scope of computational capabilities.
Indeed, while $\tau$ decreases, the function $n^{-s}\brackets{\log n}^{s(d-1)}$ grows rapidly as $n$ goes from 1 to $e^{d-1}$.
For $n^{-s}\brackets{\log n}^{s(d-1)}$ to become less than one, $n$ even has to be super exponentially large in $d$.
Thus, any estimate for the sequence $\tau$ in terms of $n^{-s}\brackets{\log n}^{s(d-1)}$ is useless
to describe its behavior in the range $n\leq 2^d$, its so called preasymptotic behavior.
As a replacement, we will prove the following estimate in Section~$\ref{preasymptoticssection}$.

\begin{thm}
\label{preasymptoticstheorem1}
Let $\sigma$ be a nonincreasing zero sequence
and $\tau$ be the nonincreasing rearrangement of its $d$th tensor power.
Let $\sigma(1)>\sigma(2)>0$ and assume that $\sigma(n) \leq C\, n^{-s}$ for some $s,C>0$ and all $n\geq 2$.
For any $n\in\set{2,\hdots,2^d}$,
\begin{equation*}
 \frac{\sigma(2)}{\sigma(1)} \cdot \brackets{\frac{1}{n}}^{\frac{\log \brackets{\sigma(1)/\sigma(2)}}{\log\brackets{1 + \frac{d}{\log_{2} n}}}}
 \leq \frac{\tau(n)}{\tau(1)} \leq
 \brackets{\frac{\exp\brackets{\brackets{C/\sigma(1)}^{2/s}}}{n}}^{\frac{\log \brackets{\sigma(1)/\sigma(2)}}{\log\brackets{\brackets{\sigma(1)/\sigma(2)}^{2/s}\, d}}}
.\end{equation*}
\end{thm}
Let us assume the power (or dimension) $d$ to be large.
Then the tensor power sequence, which roughly decays like $n^{-s}$ for huge values of $n$,
roughly decays like $n^{-t_d}$ with $t_d=\log \brackets{\sigma(1)/\sigma(2)}/ \log d$ for small values of $n$.
This is why I will refer to $t_d$
as preasymptotic rate of the tensor power sequence.
The preasymptotic rate is much worse than the asymptotic rate.
This is not an unusual phenomenon for high-dimensional problems.
Comparable estimates for the case of
$\tau$ being the sequence of approximation numbers of the embedding
$H^{s}_{\rm mix}\brackets{\mathbb{T}^d} \hookrightarrow L_2\brackets{\mathbb{T}^d}$
are established in Theorem 4.9, 4.10, 4.17 and 4.20 of \cite{ksu}.
See \cite{cw}, \cite{ksu2} or \cite{cw2} for other examples.
%... Mit dazu???
%As an illustration, we ask for the value of $a_{n,d}$, if $n$ is of the form $d^p$ for some $p\geq 0$ as $d$ tends to infinity.
%If $\sigma(1)\neq 1$, it either tends to infinity or to zero.
%But in the normalized case of $\sigma(1)=1$, it must be somewhere between $\sigma(2)^{p+1}$ and $\sigma(2)^p=n^{-t_d}$.
An interesting consequence of these preasymptotic estimates is the following tractability result.
For each $d\in\IN$, let $T_d$ be a compact norm-one operator between two Hilbert spaces
and let $T_d^d$ be its $d$th tensor power.
Assume that the corresponding approximation numbers $a_n\brackets{T_d}$ are nonincreasing in $d$
and that $a_n\brackets{T_1}$ decays polynomially in $n$.
Then the problem of approximating $T_d^d$ by linear functionals is strongly polynomially tractable,
iff it is polynomially tractable, iff $a_2\brackets{T_d}$ decays polynomially in $d$.

In Section~\ref{applicationssection}, these results will be applied to the
$L_2$-approximation of mixed order Sobolev functions on the $d$-torus, as well as
mixed order Jacobi and Sobolev functions on the $d$-cube,
taking different normalizations into account.
For instance, we will consider the 
$L_2$-embedding
\begin{equation}
 T_s^d: H^s_{\rm mix}\brackets{[0,1]^d} \hookrightarrow L_2\brackets{[0,1]^d}
\end{equation}
of the $d$-variate Sobolev space $H^s_{\rm mix}\brackets{[0,1]^d}$
with dominating mixed smoothness $s\in\IN$, equipped with the
scalar product
\begin{equation}
 \scalar{f}{g}=\sum_{\alpha\in\set{0,\hdots,s}^d} \scalar{\diff^\alpha f}{\diff^\alpha g}_{L_2}.
\end{equation}
Let $\widetilde T_s^d$ be the restriction
of $T_s^d$ to the subspace $H^s_{\rm mix}\brackets{\mathbb{T}^d}$ of periodic functions.
Theorem~\ref{asymptotic theorem} yields that the approximation numbers of these embeddings satisfy
\begin{equation}
 \lim\limits_{n\to\infty} \frac{a_n(T_s^d)\cdot n^s}{\brackets{\log n}^{s(d-1)}}
 = \lim\limits_{n\to\infty} \frac{a_n(\widetilde T_s^d)\cdot n^s}{\brackets{\log n}^{s(d-1)}}
 = \brackets{\pi^d\cdot \brackets{d-1}!}^{-s}.
\end{equation}
In particular, they do not only have the same rate of convergence,
but even the limit of their ratio is one.
This means that the $L_2$-approximation of mixed order Sobolev functions on the $d$-cube
with $n$ linear functionals is just as hard for nonperiodic functions as for periodic functions, if $n$ is large enough.
The preasymptotic rate $\tilde t_d$ for the periodic case satisfies
\begin{equation}
 \frac{s\cdot \log \brackets{2\pi}}{\log d}
 \leq \tilde t_d
 \leq \frac{s\cdot \log \brackets{2\pi}+1}{\log d}.
\end{equation}
Although this is significantly worse than the asymptotic main rate $s$,
it still grows linearly with the smoothness.
An increasing dimension can hence be neutralized by increasing the smoothness of the functions.
In contrast, the preasymptotic rate $t_d$ for the nonperiodic case satisfies
\begin{equation}
 \frac{1.2803}{\log d} \leq t_d %\leq \frac{\log \sqrt{13}}{\log d}
 \leq \frac{1.2825}{\log d}
\end{equation}
for any $s\geq 2$.
This means that increasing the smoothness of the functions beyond $s=2$ in the nonperiodic setting
is a very ineffective way of reducing the approximation error.
The $L_2$-approximation of mixed order Sobolev functions on the $d$-cube
with less than $2^d$ linear functionals is hence much harder for nonperiodic functions than for periodic functions.
This is also reflected in the corresponding tractability results:
The approximation problem $\{\widetilde T_{s_d}^d\}$
is (strongly) polynomially tractable, iff the smoothness $s_d$ grows at least logarithmically with the dimension,
whereas the approximation problem $\{T_{s_d}^d\}$
is never (strongly) polynomially tractable.
A similar effect for functions with coordinatewise increasing smoothness has already been observed by
Papageorgiou and Woźniakowski in \cite{pw}.
However, the tractability result for the space of periodic functions heavily depends on the
side length $b-a$ of the torus $\mathbb{T}^d=[a,b]^d$.
If it is less than $2\pi$, (strong) polynomial tractability is equivalent to logarithmic increase of the smoothness.
If it equals $2\pi$, (strong) polynomial tractability is equivalent to polynomial increase of the smoothness.
If it is larger than $2\pi$, there cannot be (strong) polynomial tractability.
These tractability results and interpretations can be found in Section~\ref{tracsection}.

\section{Asymptotic Behavior of Tensor Power Sequences}
\label{asymptoticssection}

Let $\sigma$ be a nonincreasing zero sequence
and $\tau$ be the nonincreasing rearrangement of its $d$th tensor power.
Fix some $s>0$ and let us consider the quantities
\begin{align*}
 &C_1=\limsup\limits_{n\to\infty}\, \sigma(n) n^s,
 &&c_1=\liminf\limits_{n\to\infty}\, \sigma(n) n^s,\\
 &C_d=\limsup\limits_{n\to\infty} \frac{\tau(n)\cdot n^s}{\brackets{\log n}^{s(d-1)}},
 &&c_d=\liminf\limits_{n\to\infty} \frac{\tau(n)\cdot n^s}{\brackets{\log n}^{s(d-1)}}.
\end{align*}
These limits may be both infinite or zero.
They can be interpreted as asymptotic or optimal constants for the bounds
\begin{align}
\label{rateupperbound}
\tau(n) &\leq C \cdot n^{-s}\brackets{\log n}^{s(d-1)}\quad\text{and}\\
\label{ratelowerbound}
\tau(n) &\geq c \cdot n^{-s}\brackets{\log n}^{s(d-1)}
.\end{align}
For any $C>C_d$ respectively $c<c_d$ there is a threshold $n_0\in\IN$ such that
(\ref{rateupperbound}) respectively (\ref{ratelowerbound}) holds for all $n\geq n_0$,
whereas for any $C<C_d$ respectively $c>c_d$ there is no such threshold.
Theorem~\ref{babenko mityagin theorem} states that $C_d$ is finite, whenever $C_1$ is finite,
whereas $c_d$ is positive, whenever $c_1$ is positive.
Theorem~\ref{asymptotic theorem} is more precise. It states that
\begin{equation}
\label{constantsrelations}
 \frac{c_1^d}{{(d-1)!}^s}
 \leq c_d \leq C_d
 \leq \frac{C_1^d}{{(d-1)!}^s}.
\end{equation}
In this section, we will give its proof.
We will also show that equality can but does not always hold.
Note that the proof provides a possibility to track down admissible thresholds $n_0$
for any $C> \frac{C_1^d}{{(d-1)!}^s}$ respectively any $c<\frac{c_1^d}{{(d-1)!}^s}$.

For the proof,
it will be essential to study the asymptotics of the cardinalities
\begin{equation}
\label{cardinality def}
 A_N(r,l)=\#\set{\boldsymbol{n}\in\set{N,N+1,\hdots}^l \mid \prod_{j=1}^l n_j \leq r}
\end{equation}
for $l\in\set{1,\hdots,d}$ and $N\in\IN$ as $r\to \infty$.
In \cite[Lemma 3.2]{ksu}, it is shown that
\begin{equation}
\label{cardinalityasymptoticsexplicitformula}
 r \brackets{\frac{\brackets{\log\frac{r}{2^l}}^{l-1}}{\brackets{l-1}!} - \frac{\brackets{\log\frac{r}{2^l}}^{l-2}}{\brackets{l-2}!}}
 \leq A_2(r,l) \leq
 r \frac{\brackets{\log r}^{l-1}}{\brackets{l-1}!}
\end{equation}
for $l\geq 2$ and $r\in\set{4^l,4^l+1,\hdots}$, see also \cite[Theorem~3.4]{cd}. Consequently, we have
\begin{equation}
\label{cardinalityasymptoticsformula}
 \lim\limits_{r\to\infty} \frac{A_N(r,l)}{r \brackets{\log r}^{l-1}} = \frac{1}{\brackets{l-1}!}
\end{equation}
for $N=2$. In fact, (\ref{cardinalityasymptoticsformula}) holds true for any $N\in\IN$.
This can be derived from the case $N=2$,
but for the reader's convenience, I will give a complete proof.

\begin{lemma}
 \label{cardinalityasymptotics}
 \begin{equation*}
  \lim\limits_{r\to\infty} \frac{A_N(r,l)}{r \brackets{\log r}^{l-1}} = \frac{1}{\brackets{l-1}!}.
 \end{equation*}
\end{lemma}

\begin{proof}
 Note that for all values of the parameters,
 \begin{equation}
  A_N(r,l+1) = \sum\limits_{k=N}^\infty A_N\brackets{\frac{r}{k},l},
 \end{equation}
 where $A_N\brackets{\frac{r}{k},l}=0$ for $k > \frac{r}{N^l}$.
 This allows a proof by induction on $l\in\IN$.
 
 Like in estimate (\ref{cardinalityasymptoticsexplicitformula}), we first show that
 \begin{equation}
 \label{cardinalitysharpupperbound}
   A_2(r,l) \leq r \frac{\brackets{\log r}^{l-1}}{\brackets{l-1}!}
 \end{equation}
 for any $l\in\IN$ and $r\geq 1$. This is obviously true for $l=1$.
 On the other hand, if this relation holds for some $l\in\IN$ and if $r\geq 1$, then
 \begin{equation}
  \begin{split}
   &A_2(r,l+1)
   = \sum\limits_{k=2}^{\left\lfloor r\right\rfloor} A_2\brackets{\frac{r}{k},l}
   \leq \sum\limits_{k=2}^{\left\lfloor r \right\rfloor} \frac{r \brackets{\log \frac{r}{k}}^{l-1}}{k \brackets{l-1}!}\\
   &\leq \frac{r}{\brackets{l-1}!} \int_1^r \frac{\brackets{\log \frac{r}{x}}^{l-1}}{x} ~\d x
   = \frac{r}{\brackets{l-1}!} \left[-\frac{1}{l}\brackets{\log \frac{r}{x}}^l \right]_1^r
   = r \frac{\brackets{\log r}^l}{l!}
  \end{split}
 \end{equation}
 and (\ref{cardinalitysharpupperbound}) is proven. In particular, we have
 \begin{equation}
 \label{cardinalitylimsup}
  \limsup\limits_{r\to\infty} \frac{A_N(r,l)}{r \brackets{\log r}^{l-1}} \leq \frac{1}{\brackets{l-1}!}
 \end{equation}
 for $l\in\IN$ and $N=2$. Clearly, the same holds for $N\geq 2$, since $A_N(r,l)$ is decreasing in $N$.
 Relation (\ref{cardinalitylimsup}) for $N=1$ follows from the case $N=2$ by the identity
 \begin{equation}\begin{split}
  A_1(r,l) 
  &= \sum\limits_{m=0}^l \#\set{\boldsymbol{n}\in\IN^l
  \mid \#\set{1\leq j\leq l \mid n_j\neq 1}=m \land \prod_{j=1}^d n_j \leq r}\\
  &= \mathbf{1}_{r\geq 1} + \sum\limits_{m=1}^l \binom{l}{m}\cdot A_2(r,m)
 .\end{split}\end{equation}
 It remains to prove
 \begin{equation}
 \label{cardinalityliminf}
  \liminf\limits_{r\to\infty} \frac{A_N(r,l)}{r \brackets{\log r}^{l-1}} \geq \frac{1}{\brackets{l-1}!}
 \end{equation}
 for $N\in\IN$ and $l\in\IN$. Again, this is obvious for $l=1$.
 Suppose, (\ref{cardinalityliminf}) holds for some $l\in\IN$ and let $b<1$.
 Then there is some $r_0\geq 1$ such that
 \begin{equation}
  A_N(r,l) \geq b r \frac{\brackets{\log r}^{l-1}}{\brackets{l-1}!}
 \end{equation}
 for all $r\geq r_0$ and hence
 \begin{equation}
  \begin{split}
   &A_N(r,l+1)
   \geq \sum\limits_{k=N}^{\left\lfloor r/r_0\right\rfloor} A_N\brackets{\frac{r}{k},l}
   \geq \sum\limits_{k=N}^{\left\lfloor r/r_0\right\rfloor} \frac{b r \brackets{\log \frac{r}{k}}^{l-1}}{k \brackets{l-1}!}\\
   &\geq \frac{b r}{\brackets{l-1}!} \int_N^{\frac{r}{r_0}} \frac{\brackets{\log \frac{r}{x}}^{l-1}}{x} ~\d x
   = \frac{b r}{l!} \brackets{\brackets{\log \frac{r}{N}}^l- \brackets{\log r_0}^l}
   \geq b^2 r \frac{\brackets{\log r}^l}{l!} 
  \end{split}
 \end{equation}
 for large $r$. Since this is true for any $b<1$, the induction step is complete.
\end{proof}

\begin{proof}[Proof of Theorem~\ref{asymptotic theorem}]
 Without loss of generality, we can assume that $s=1$ and $\sigma(1)=1$.
 If $\sigma(1)\neq 0$, the stated inequalities follow from the corresponding
 inequalities for the sequence $\tilde{\sigma}=\brackets{\sigma/\sigma(1)}^{1/s}$.
 If $\sigma(1)=0$, they are trivial.
 
 Proof of $(i)$: Let $c_3>c_2>c_1>c$.
 There is some $N\in\IN$ such that for any $n\geq N$, we have
 \begin{equation}
 \label{asymptotictheoremupperbounddim1}
  \sigma(n)\leq c_1\, n^{-1}.
 \end{equation}
 We want to prove
 \begin{equation}
 \label{asymptotictheoremupperbound}
  \limsup\limits_{n\to\infty}  \frac{\tau(n)\,n}{\brackets{\log n}^{d-1}} \leq \frac{c^d}{(d-1)!}.
 \end{equation}
 Since $n/\brackets{\log n}^{d-1}$ is finally increasing,
 instead of giving an upper bound for $\tau(n)$ in terms of $n$,
 we can just as well give an upper bound for $n$ in terms of $\tau(n)$ to obtain (\ref{asymptotictheoremupperbound}).
 Clearly, there are at least $n$ elements in the tensor power sequence greater than or equal to $\tau(n)$ and hence
 \begin{equation}
  \begin{split}
   n &\leq \#\set{\boldsymbol{n}\in\IN^d \mid \sigma_d(\boldsymbol{n}) \geq \tau(n)}\\
   &= \sum\limits_{l=0}^d \#\set{\boldsymbol{n}\in\IN^d
   \mid \#\set{1\leq j \leq d \mid n_j\geq N}=l \land \sigma_d(\boldsymbol{n}) \geq \tau(n)}\\
   &\overset{\sigma(1)=1}{\leq} \sum\limits_{l=0}^d \binom{d}{l} N^{d-l}\, \#\set{\boldsymbol{n}\in\set{N,N+1,\hdots}^l
   \mid \sigma_d(\boldsymbol{n}) \geq \tau(n)}
  .\end{split}
 \end{equation}
 For every $\boldsymbol{n}$ in the last set, relation (\ref{asymptotictheoremupperbounddim1}) implies
 that $\prod_{j=1}^d n_j \leq c_1^l\,\tau(n)^{-1}$. Thus,
 \begin{equation}
   n\leq \sum\limits_{l=0}^d \binom{d}{l}\, N^{d-l}\, A_N\brackets{c_1^l\,\tau(n)^{-1},l}.
 \end{equation}
 Lemma~\ref{cardinalityasymptotics} yields that,
 if $n$ and hence $c_1^l\,\tau(n)^{-1}$ is large enough,
 \begin{equation}
  A_N\brackets{c_1^l\,\tau(n)^{-1},l}
  \leq \frac{c_2^l\,\tau(n)^{-1}}{\brackets{l-1}!} \brackets{\log \brackets{c_2^l\,\tau(n)^{-1}}}^{l-1}
 \end{equation}
 for $l\in\set{1,\hdots,d}$. Letting $n\to\infty$, the term for $l=d$ is dominant and hence
 \begin{equation}
  n \leq \frac{c_3^d\,\tau(n)^{-1}}{\brackets{d-1}!} \brackets{\log \brackets{c_3^d\,\tau(n)^{-1}}}^{d-1}
 \end{equation}
 for large values of $n$. By the monotonicity of $n/\brackets{\log n}^{d-1}$, we obtain
 \begin{equation}
  \frac{\tau(n)\,n}{\brackets{\log n}^{d-1}}
  \leq \frac{c_3^d}{\brackets{d-1}!} \cdot
  \brackets{\frac{\log \brackets{c_3^d \tau(n)^{-1}}}{\log \brackets{\tau(n)^{-1}\cdot \frac{c_3^d}{\brackets{d-1}!} \brackets{\log \brackets{c_3^d \tau(n)^{-1}}}^{d-1}}}}^{d-1}
 .\end{equation}
 The fraction in brackets tends to one as $n$ and hence $\tau(n)^{-1}$ tends to infinity and thus
 \begin{equation}
  \limsup\limits_{n\to\infty} \frac{\tau(n)\,n}{\brackets{\log n}^{d-1}} \leq \frac{c_3^d}{\brackets{d-1}!}
 .\end{equation}
 Since this is true for any $c_3>c$, the proof of (\ref{asymptotictheoremupperbound}) is complete.
 
 Proof of $(ii)$: Let $0<c_3<c_2<c_1<c$. There is some $N\in\IN$ such that for any $n\geq N$,
 we have
 \begin{equation}
 \label{asymptotictheoremlowerbounddim1}
  \sigma(n)\geq c_1\, n^{-1}.
 \end{equation}
 We want to prove
 \begin{equation}
 \label{asymptotictheoremlowerbound}
  \liminf\limits_{n\to\infty} \frac{\tau(n)\,n}{\brackets{\log n}^{d-1}} \geq \frac{c^d}{{(d-1)!}^s}
 \end{equation}
 for any $d\in\IN$. Clearly, there are at most $n-1$ elements in the tensor power sequence greater than $\tau(n)$ and hence
 \begin{equation}
  n > \#\set{\boldsymbol{n}\in\IN^d \mid \sigma_d(\boldsymbol{n}) > \tau(n)}
  \geq \#\set{\boldsymbol{n}\in\set{N,N+1,\hdots}^d \mid \sigma_d(\boldsymbol{n}) > \tau(n)}
 .\end{equation}
 Relation (\ref{asymptotictheoremlowerbounddim1}) implies that every $\boldsymbol{n}\in\set{N,N+1,\hdots}^d$
 with $\prod_{j=1}^d n_j < c_1^d\, \tau(n)^{-1}$ is contained in the last set.
 This observation and Lemma~\ref{cardinalityasymptotics} yield that
 \begin{equation}
  n > A_N\brackets{c_2^d\, \tau(n)^{-1},d} \geq \frac{c_3^d\, \tau(n)^{-1}}{\brackets{d-1}!} \brackets{\log \brackets{c_3^d\, \tau(n)^{-1}}}^{d-1}
 \end{equation}
 for sufficiently large $n$. By the monotonicity of $n/\brackets{\log n}^{d-1}$ for large $n$, we obtain
 \begin{equation}
  \frac{\tau(n)\,n}{\brackets{\log n}^{d-1}}
  \geq \frac{c_3^d}{\brackets{d-1}!} \cdot \brackets{\frac{\log \brackets{c_3^d \tau(n)^{-1}}}{\log \brackets{\frac{c_3^d}{\brackets{d-1}!} \brackets{\log \brackets{c_3^d \tau(n)^{-1}}}^{d-1} \tau(n)^{-1} }}}^{d-1}
 .\end{equation}
 The fraction in brackets tends to one as $n$ and hence $\tau(n)^{-1}$ tends to infinity and thus
 \begin{equation}
  \liminf\limits_{n\to\infty} \frac{\tau(n)\,n}{\brackets{\log n}^{d-1}}
  \geq \frac{c_3^d}{\brackets{d-1}!}
 .\end{equation}
 Since this is true for any $c_3<c$, the proof of (\ref{asymptotictheoremlowerbound}) is complete.
\end{proof}

This proves the relations~(\ref{constantsrelations}) of the asymptotic constants.
Obviously, there must be equality in all these relations,
if the limit of $\sigma(n)\,n^s$ for $n\to\infty$ exists.
It is natural to ask, whether any of these equalities always holds true.
The answer is no, as shown by the following example.

\begin{ex}
 The sequence $\sigma$, defined by $\sigma(n)=2^{-k}$ for $n\in\set{2^k,\hdots,2^{k+1}-1}$ and $k\in\IN_0$, decays linearly in $n$, but is constant on segments of length $2^k$.
 It satisfies
 \begin{equation}
  C_1=\limsup\limits_{n\to\infty} \sigma(n) n =\lim\limits_{k\to\infty} 2^{-k}\cdot\brackets{2^{k+1}-1}=2
 \end{equation}
 and
 \begin{equation}
  c_1=\liminf\limits_{n\to\infty} \sigma(n) n =\lim\limits_{k\to\infty} 2^{-k}\cdot 2^k=1
 .\end{equation}
 Also the values of the nonincreasing rearrangement $\tau$ of its $d$th tensor power
 are of the form $2^{-k}$ for some $k\in\IN_0$, where
 \begin{equation}
  \begin{split}
  &\#\set{n\in\IN \mid \tau(n)=2^{-k}}
  = \sum\limits_{\abs{\boldsymbol{k}}=k} \#\set{\boldsymbol{n}\in\IN^d\mid \sigma(n_j)=2^{-k_j} \text{ for } j=1\hdots d}\\
  &= \sum\limits_{\abs{\boldsymbol{k}}=k} 2^k
  = 2^k \cdot \binom{k+d-1}{d-1}
  = \frac{2^k}{(d-1)!}\cdot (k+1)\cdot\hdots\cdot(k+d-1)
  .\end{split}
 \end{equation}
 Hence, $\tau(n)=2^{-k}$ for $N(k-1,d)< n\leq N(k,d)$ with $N(-1,d)=0$ and
 \begin{equation}
  N(k,d)= \sum\limits_{j=0}^k \frac{2^j}{(d-1)!}\cdot (j+1)\cdot\hdots\cdot(j+d-1)
 \end{equation}
 for $k\in\IN_0$. The monotonicity of $n / \brackets{\log n}^{d-1}$ for large $n$ implies
 \begin{equation}
  \label{limsupinfexample1}
   C_d=\limsup\limits_{n\to\infty} \frac{\tau(n)\cdot n}{\brackets{\log n}^{d-1}} = \lim\limits_{k\to\infty} \frac{2^{-k}\cdot N(k,d)}{\brackets{\log N(k,d)}^{d-1}}
  \end{equation}
  and
  \begin{equation}
  \label{limsupinfexample2}
   c_d=\liminf\limits_{n\to\infty} \frac{\tau(n)\cdot n}{\brackets{\log n}^{d-1}} = \lim\limits_{k\to\infty} \frac{2^{-k}\cdot N(k-1,d)}{\brackets{\log N(k-1,d)}^{d-1}}
 .\end{equation}
 We insert the relations
 \begin{equation}
   N(k,d) \leq \frac{(k+d)^{d-1}}{(d-1)!}\sum\limits_{j=0}^k 2^j \leq \frac{2^{k+1}\cdot(k+d)^{d-1}}{(d-1)!}
 \end{equation}
 and
 \begin{equation}
   N(k,d) \geq \frac{(k-l)^{d-1}}{(d-1)!}\sum\limits_{j=k-l+1}^k 2^j = \frac{2^{k+1}(k-l)^{d-1}}{(d-1)!}\brackets{1-2^{-l}}
 \end{equation}
 for arbitrary $l\in\IN$ in (\ref{limsupinfexample1}) and (\ref{limsupinfexample2}) and obtain
  \begin{equation}
   C_d = 2\cdot\frac{\brackets{\log_2 e}^{d-1}}{(d-1)!} \quad\quad\text{and}\quad\quad
   c_d = \frac{\brackets{\log_2 e}^{d-1}}{(d-1)!}.
 \end{equation}
 In particular, %for $d\neq 1$ we have a chain of strict inequalities:
 \begin{equation}
  \frac{c_1^d}{\brackets{d-1}!}
  < c_d
  < C_d
  < \frac{C_1^d}{\brackets{d-1}!}
 \quad\text{for } d\neq 1.\end{equation}
\end{ex}

More generally, the tensor product of $d$ nonincreasing zero sequences $\sigma^{(j)}:\IN\to\IR$
is the sequence $\sigma_d:\IN^d\to\IR$,
where $\sigma_d(n_1,\dots,n_d)=\prod_{j=1}^d\sigma^{(j)}(n_j)$.
It can be rearranged to a nonincreasing zero sequence $\tau$.
An example of such a sequence is given by the $L_2$-approximation numbers of
Sobolev functions on the $d$-torus with mixed order $(s_1,\hdots,s_d)\in\IR_+^d$.
They are generated by the $L_2$-approximation numbers of the univariate Sobolev spaces $H^{s_j}\brackets{\mathbb{T}}$,
which are of order $n^{- s_j}$.
It is known that $\tau$
has the order $n^{-s}\brackets{\log n}^{s(l-1)}$ in this case,
where $s$ is the minimum among all numbers $s_j$ and $l$ is its multiplicity.
This was proven by Mityagin \cite{mityagin} for integer vectors $(s_1,\hdots,s_d)$ and by Nikol’skaya \cite{nikolskaya} in the general case.
See \cite[pp.\,32, 36, 72]{temlyakov} and \cite{dtu} for more details.
It is not hard to deduce that
the order of decay of $\tau$ is at least (at most) $n^{-s}\brackets{\log n}^{s(l-1)}$,
whenever the order of the factor sequences $\sigma^{(j)}$
is at least (at most) $n^{-s_j}$.
But in contrast to the tensor power case, asymptotic constants of tensor product sequences in general
are not determined by the asymptotic constants of the factor sequences.

\begin{ex}
 Consider the sequences $\sigma,\mu,\tilde\mu: \IN\to\IR$ with
 \begin{equation}
  \sigma(n)=n^{-1}, \quad \mu(n)=n^{-2}, \quad \tilde\mu(n)=
  \left\{\begin{array}{lr}
        1, & \text{for } n\leq N,\\
        n^{-2}, & \text{for } n>N,
        \end{array}\right.
 \end{equation}
for some $N\in\IN$.
The tensor product $\sigma_2:\IN^2\to\IR$ of $\sigma$ and $\mu$ has the form
\begin{equation}
 \sigma_2(n_1,n_2)= n_1^{-1} n_2^{-2}
\end{equation}
and its nonincreasing rearrangement $\tau$ satisfies for all $n\in\IN$ that
\begin{equation}
 \begin{split}
  n &\leq \#\set{(n_1,n_2)\in \IN^2 \mid \sigma_2(n_1,n_2) \geq \tau(n)}
  = \#\set{(n_1,n_2) \mid {n}_1 n_2^2 \leq \tau(n)^{-1}}\\
  &\leq \sum\limits_{ n_2=1}^\infty \#\set{ n_1\in\IN \mid  n_1 \leq a_n^{-1}  n_2^{-2}}
  \leq \tau(n)^{-1} \sum\limits_{ n_2=1}^\infty n_2^{-2}
  \leq 2 \tau(n)^{-1},
 \end{split}
\end{equation}
and hence
\begin{equation}
 \limsup\limits_{n\to\infty} \tau(n)n \leq 2.
\end{equation}
The tensor product $\tilde\sigma_2:\IN^2\to\IR$ of $\sigma$ and $\tilde\mu$ takes the form
\begin{equation}
 \tilde \sigma_2(n_1,n_2)=
 \left\{\begin{array}{lr}
        n_1^{-1}, & \text{if } n_2\leq N,\\
        n_1^{-1} n_2^{-2}, & \text{else},
        \end{array}\right.
\end{equation}
and its nonincreasing rearrangement $\tilde\tau$ satisfies for all $n\in\IN$ that
\begin{equation}
\begin{split}
  n &\geq \#\set{(n_1,n_2)\in\IN^2 \mid\tilde \sigma_2(n_1,n_2) >\tilde a_n}
  \geq N \#\set{n_1\in\IN \mid n_1^{-1} >\tilde \tau(n)}\\
  &\geq N \brackets{\tilde \tau(n)^{-1} -1}
\end{split}
\end{equation}
and thus
\begin{equation}
 \liminf\limits_{n\to\infty}\tilde \tau(n) n \geq N.
\end{equation}
Hence, matching asymptotic constants of the factor sequences do not necessarily lead to
matching asymptotic constants of the tensor product sequences.

\end{ex}

\section{Preasymptotic Behavior of Tensor Power Sequences}
\label{preasymptoticssection}

In order to estimate the size of $\tau(n)$ for small values of $n$,
we give explicit estimates for $A_2(r,l)$ from (\ref{cardinality def})
for $l\leq d$ and small values of $r$.
The right asymptotic behavior of these estimates, however, is less important.
Note that $A_2(r,l)=0$ for $r< 2^l$.

\begin{lemma}
\label{hyperboliccrosscountinglemma}
Let $r\geq 0$ and $l\in\IN$. For any $\delta>0$ we have
\begin{align*}
 &A_2(r,l) \leq \frac{r^{1+\delta}}{\delta^{l-1}}  	&&\text{and}\\
 &A_2(r,l) \geq \frac{r}{3\cdot 2^{l-1}}			&&\text{for } r\geq 2^l.
\end{align*}
\end{lemma}

\begin{proof}
 Both estimates hold in the case $l=1$, since
 \begin{equation}
  A_2(r,1)
  = \left\{\begin{array}{lr}
        0, & \text{for } r< 2,\\
        \lfloor r \rfloor -1, & \text{for } r\geq 2
        .\end{array}\right.
 \end{equation}
If they hold for some $l\in\IN$, then
\begin{equation}
 \begin{split}
  A_2(r,l+1)
  &= \sum\limits_{k=2}^\infty A_2\left(\frac{r}{k},l\right)
  \leq\ \frac{r^{1+\delta}}{\delta^{l-1}} \sum\limits_{k=2}^\infty \frac{1}{k^{1+\delta}}\\
  &\leq\ \frac{r^{1+\delta}}{\delta^{l-1}} \int_1^\infty \frac{1}{x^{1+\delta}} ~\d x
  =\ \frac{r^{1+\delta}}{\delta^l}
 \end{split}
\end{equation}
and for $r\geq 2^{l+1}$
\begin{equation}
  A_2(r,l+1)
  \geq A_2\brackets{\frac{r}{2},l}
  \geq \frac{r/2}{3\cdot 2^{l-1}}
  = \frac{r}{3\cdot 2^l}.
\end{equation}
We have thus proven Lemma~\ref{hyperboliccrosscountinglemma} by induction.
\end{proof}

\begin{thm}
\label{preasymptoticstheorem}
Let $\sigma$ be a nonincreasing zero sequence with $1=\sigma(1)>\sigma(2)>0$
and let $\tau$ be the nonincreasing rearrangement of its $d$th tensor power.
\begin{itemize}
 \item[(i)] Suppose that $\sigma(n) \leq C\, n^{-s}$ for some $s,C>0$ and all $n\geq 2$ and let $\delta\in (0,1]$.
 For any $n\in\IN$,
 \begin{equation*}
  \tau(n) \leq\, \brackets{\frac{\tilde{C}(\delta)}{n}}^{\alpha(d,\delta)},
  \quad\text{where}
  \end{equation*}
  \begin{equation*}                 
  \tilde{C}(\delta)=\exp{\brackets{\frac{C^{(1+\delta)/s}}{\delta}}} \quad\text{and}\quad
  \alpha(d,\delta)=\frac{\log\sigma(2)^{-1}}{\log \brackets{\sigma(2)^{-(1+\delta)/s}\cdot d}}>0
  .\end{equation*}
 \item[(ii)] Let $v=\#\set{n \geq 2\mid \sigma(n)= \sigma(2)}$. For any $n\in\set{2,\hdots,(1+v)^d}$,
 \begin{equation*}
   \tau(n) \geq\, \sigma(2) \cdot \brackets{\frac{1}{n}}^{\beta(d,n)},
  \quad\text{where}\quad
  \beta(d,n) = \frac{\log \sigma(2)^{-1}}{\log \brackets{1+\frac{v}{\log_{1+v} n}\cdot d}}>0.
 \end{equation*}
\end{itemize}
\end{thm}

The assumption $\sigma(1)=1$ merely reduces the complexity of the estimates.
We can easily translate the above estimates for arbitrary $\sigma(1)>\sigma(2)>0$ by applying
Theorem~\ref{preasymptoticstheorem} to the sequence $\brackets{\sigma(n)/\sigma(1)}_{n\in\IN}$.
We simply have to replace $\sigma(2)$ by $\sigma(2)/\sigma(1)$, $C$ by $C/\sigma(1)$
and $\tau(n)$ by $\tau(n)/\sigma(1)^d$.
Theorem $\ref{preasymptoticstheorem1}$, as stated in the introduction,
is an immediate consequence of Theorem $\ref{preasymptoticstheorem}$.
Obviously, $\sigma(2)=\sigma(1)$ implies $\tau(n)=\sigma(1)^d$ for every $n\leq (1+v)^d$,
whereas $\sigma(2)=0$ implies $\tau(n)=0$ for every $n\geq 2$.

\begin{proof}
Part $(i)$:
Let $n\in\IN$. There is some $L\geq 0$ with $\tau(n)=\sigma(2)^L$.
If $\sigma_d(\boldsymbol{n}) \geq \tau(n)$, the number $l$ of components of $\boldsymbol{n}$ not equal to one is at most $\lfloor L\rfloor$ and hence
\begin{equation}
\label{nononebasicupperbound}
 \begin{split}
  n &\leq \#\set{\boldsymbol{n}\in\IN^d \mid \sigma_d(\boldsymbol{n}) \geq \tau(n)}\\
  &= \sum\limits_{l=0}^{\min\set{\lfloor L\rfloor,d}}  \#\set{\boldsymbol{n}\in\IN^d
  \mid \#\set{1\leq j \leq d \mid n_j\neq 1}=l \land \sigma_d(\boldsymbol{n}) \geq \tau(n)}\\
  &= 1+ \sum\limits_{l=1}^{\min\set{\lfloor L\rfloor,d}} \binom{d}{l} \#\set{\boldsymbol{n}\in\set{2,3,\hdots}^l
  \mid \sigma_d(\boldsymbol{n}) \geq \tau(n)}
 .\end{split}
\end{equation}
Since $\sigma_d(\boldsymbol{n})\leq C^l\, \prod_{j=1}^l n_j^{-s}$ for $\boldsymbol{n}\in\set{2,3,\hdots}^l$,
Lemma~\ref{hyperboliccrosscountinglemma} yields for $l\leq\min\set{\lfloor L\rfloor,d}$,
\begin{equation}
\begin{split}
  \#\set{\boldsymbol{n}\in\set{2,3,\hdots}^l \mid \sigma_d(\boldsymbol{n}) \geq \tau(n)}
  &\leq A_2\brackets{C^{l/s} \tau(n)^{-1/s},l}\\
  &\leq C^{(1+\delta)l/s} \tau(n)^{-(1+\delta)/s} \delta^{-l}
  \end{split}
\end{equation}
Obviously,
\begin{equation}
 1\leq \binom{d}{0}\cdot C^{0/s} \tau(n)^{-(1+\delta)/s}\delta^{0}.
\end{equation}
Inserting these bounds in (\ref{nononebasicupperbound}) yields
\begin{equation}
 \begin{split}
  n &\leq \sum\limits_{l=0}^{\min\set{\lfloor L\rfloor,d}} \binom{d}{l}\cdot C^{(1+\delta)l/s} \tau(n)^{-(1+\delta)/s} \delta^{-l}
  \leq \tau(n)^{-(1+\delta)/s} \sum\limits_{l=0}^{\min\set{\lfloor L\rfloor,d}} \frac{d^l}{l!} C^{(1+\delta)l/s} \delta^{-l}\\
  &\leq \sigma(2)^{-(1+\delta)L/s} d^L \sum\limits_{l=0}^{\min\set{\lfloor L\rfloor,d}} \frac{\brackets{\frac{C^{(1+\delta)/s}}{\delta}}^l}{l!}
  \leq \brackets{\sigma(2)^{-(1+\delta)/s}\cdot d}^L \exp{\brackets{\frac{C^{(1+\delta)/s}}{\delta}}}
 \end{split}
\end{equation}
and hence
\begin{equation}
 L \geq \frac{\log n - \frac{C^{(1+\delta)/s}}{\delta}}{\log \brackets{\sigma(2)^{-(1+\delta)/s}\cdot d}}.
\end{equation}
Thus
\begin{equation}
\label{otherformofpreasymptotics1}
\tau(n)= \sigma(2)^L
\leq \exp\brackets{\frac{\brackets{\frac{C^{(1+\delta)/s}}{\delta} -\log n} \log\sigma(2)^{-1}}{\log \brackets{\sigma(2)^{-(1+\delta)/s}\cdot d}}}
= \brackets{\frac{\exp{\brackets{\frac{C^{(1+\delta)/s}}{\delta}}}}{n}}^{\alpha(d,\delta)}
\end{equation}
with
\begin{equation}
 \alpha(d,\delta)=\frac{\log\sigma(2)^{-1}}{\log \brackets{\sigma(2)^{-(1+\delta)/s}\cdot d}}.
\end{equation}

Part $(ii)$:
Let $n\in\set{2,\hdots,(1+v)^d}$. Then $\sigma(2)^d\leq \tau(n)\leq \sigma(2)$.
If $\tau(n)$ equals $\sigma(2)$, the lower bound is trivial.
Else, there is some $L\in\{1,\hdots,d-1\}$ such that $\tau(n)\in [\sigma(2)^{L+1},\sigma(2)^L)$.
Clearly,
\begin{equation}
\label{nononebasiclowerbound}
n > \#\set{\boldsymbol{n}\in\IN^d \mid \sigma_d(\boldsymbol{n})> \tau(n) }
\geq \sum_{l=1}^L \binom{d}{l} \#\set{\boldsymbol{n}\in\set{2,3,\hdots}^l \mid \sigma_d(\boldsymbol{n})> \tau(n)}.
\end{equation}
If $l\leq L$, we have $\sigma_d(\boldsymbol{n})> \tau(n)$ for every $\boldsymbol{n}\in\set{2,\hdots,1+v}^l$ and hence
\begin{equation}
 \label{lowerboundonn}
 n \geq \sum_{l=0}^L \binom{d}{l}\, v^l
 \geq \sum_{l=0}^L \binom{L}{l} \brackets{\frac{d}{L}}^l v^l
 =\brackets{1+\frac{vd}{L}}^L.
\end{equation}
Since $d/L$ is bigger than one, this yields in particular that
\begin{equation}
 \label{preestimateL}
 L \leq \log_{1+v} n.
\end{equation}
We insert this auxiliary estimate on $L$ in (\ref{lowerboundonn}) and get
\begin{equation}
 n \geq \brackets{1+\frac{vd}{\log_{1+v} n}}^L,
\end{equation}
or equivalently
\begin{equation}
 L \leq \frac{\log n}{\log \brackets{1+\frac{vd}{\log_{1+v} n}}}.
\end{equation}
We recall that $\tau(n)\geq \sigma(2)^{L+1}$ and realize that the proof is finished.
\end{proof}

The bounds of Theorem~\ref{preasymptoticstheorem} are very explicit, but complex.
One might be bothered by the dependence of the exponent in the lower bound on $n$.
This can be overcome, if we restrict the lower bound to the case $n \leq (1+v)^{d^a}$ for some $0<a<1$
and replace $\beta(d,n)$ by 
\begin{equation}
 \tilde{\beta}(d) = \frac{\log \sigma(2)^{-1}}{\log \brackets{1+v\cdot d^{1-a}}}.
\end{equation}
Of course, we throw away information this way.
Similarly, we get a worse but still valid estimate, if we replace $v$ by one.
Note that these lower bounds are valid for any zero sequence $\sigma$,
independent of its rate of convergence.

The constants 1, $\sigma(2)$ and $\tilde C(\delta)$ are independent of the power $d$.
The additional parameter $\delta$ in the upper bound was introduced to
maximize the exponent $\alpha(d,\delta)$. 
If $\delta$ tends to zero, $\alpha(d,\delta)$ gets bigger, but also the constant $\tilde C(\delta)$ explodes.

For large values of $d$ and if $n$ is significantly smaller than $(1+v)^d$,
the exponents in both the upper and the lower bound are close to $t_d=\frac{\log \brackets{\sigma(2)/\sigma(1)}^{-1}}{\log d}$.
In other words, the sequence $\tau$ preasymptotically roughly decays like $n^{-t_d}$.

These kinds of estimates are also closely related to those in \cite[Section 3]{gw}.
Using the language of generalized tractability, Gnewuch and Woźniakowski show that the supremum of all
$p>0$ such that there is a constant $C>0$ with
\begin{equation}
 \tau(n)\leq e\cdot \brackets{C/n}^{\frac{p}{1+\log d}}
\end{equation}
for all $n\in\IN$ and $d\in\IN$ is $\min\set{s,\log \sigma(2)^{-1}}$.

\section{Applications to some Tensor Power Operators}
\label{applicationssection}

Let $X$ and $Y$ be Hilbert spaces and let $T:X\to Y$ be a compact linear operator.
The $n$th approximation number of $T$ is the quantity
\begin{equation}
\label{appnumbersdefinition}
 a_n(T) = \inf\limits_{\rank (A) < n} \norm{T-A}.
\end{equation}
It measures the power of approximating $T$ in $\mathcal{L}\brackets{X,Y}$ by operators of rank less than $n$.
Obviously, the first approximation number of $T$ coincides with its norm.
Since $W=T^*T\in\mathcal{L}(X)$ is positive semi-definite and compact,
it admits a finite or countable orthonormal basis $\mathcal{B}$ of $N(T)^\perp$ consisting of eigenvectors
$b\in\mathcal{B}$ to eigenvalues
\begin{equation}
 \lambda(b) = \scalar{Wb}{b}_X= \norm{T b}_Y^2 > 0
.\end{equation}
I will refer to $\mathcal{B}$ as the orthonormal basis associated with $T$.
It can be characterized as the orthonormal basis of $N(T)^\perp$ whose image is an orthogonal basis of $\overline{R(T)}$.
It is unique up to the choice of orthonormal bases in the finite-dimensional eigenspaces of $W$.
Clearly,
\begin{equation}
 Tf = \sum_{b\in\mathcal{B}} \scalar{f}{b}_{X} Tb \quad\text{for } f\in X.
\end{equation}
The square-roots of the eigenvalues of $W$ are called singular values of $T$.
Let $\sigma(n)$ be the $n$th largest singular value of $T$, provided $n\leq\abs{\mathcal{B}}$. Else, let $\sigma(n)=0$.
The algorithm
\begin{equation}
 A_n f= \sum_{b\in\mathcal{B}_n} \scalar{f}{b}_{X} Tb \quad\text{for } f\in X
\end{equation}
is an optimal approximation of $T$ by operators of rank less than $n$,
if $\mathcal{B}_n$ consists of all $b\in\mathcal{B}$ with $\norm{T b}_Y>\sigma(n)$.
In particular, $a_n(T)$ and $\sigma(n)$ coincide and
\begin{equation}
\label{minmax}
 a_n(T) = \min\limits_{\substack{V\subseteq X\\ \dim(V)\leq n-1}}\, \max\limits_{\substack{f\perp V\\ \norm{f}_X=1}} \norm{Tf}_Y.
\end{equation}
% Note that all $s$-numbers coincide in this setting.

We are concerned with the approximation numbers of tensor power operators, defined as follows.
Let $G$ be a set and $G^d$ be its $d$-fold Cartesian product and let $\IK\in\set{\IR,\IC}$.
The tensor product of $\IK$-valued functions $f_1,\hdots,f_d$ on $G$ is the function
\begin{equation}
 f_1\otimes\hdots\otimes f_d:\quad G^d \to \IK, \quad x \mapsto f_1(x_1)\cdot\hdots\cdot f_d(x_d).
\end{equation}
If $X$ is a Hilbert space of $\IK$-valued functions on $G$,
its $d$th tensor power $X^d$ is the smallest Hilbert space of $\IK$-valued functions on $G^d$
that contains any tensor product of functions in $X$ and satisfies
\begin{equation}
 \scalar{f_1\otimes\hdots\otimes f_d}{g_1\otimes\hdots\otimes g_d}
 = \scalar{f_1}{g_1}\cdot\hdots\cdot\scalar{f_d}{g_d}
\end{equation}
for any choice of functions $f_1,\hdots,f_d$ and $g_1,\hdots,g_d$ in $X$.
Let $Y$ be another Hilbert space of $\IK$-valued functions and let $T\in\mathcal{L}(X,Y)$.
The $d$th tensor power of $T$ is the unique operator $T^d\in\mathcal{L}(X^d,Y^d)$ that satisfies
\begin{equation}
 T^d\brackets{f_1\otimes\hdots\otimes f_d} = Tf_1\otimes\hdots\otimes Tf_d
\end{equation}
for any choice of functions $f_1,\hdots,f_d$ in $X$. If $T$ is compact, then so is $T^d$.
Moreover, if $\mathcal{B}$ is the orthonormal basis associated with $T$, then
\begin{equation}
 \mathcal{B}^d = \set{b_1\otimes\hdots\otimes b_d \mid b_1,\hdots,b_d\in\mathcal{B}}
\end{equation}
is the orthonormal basis associated with $T^d$.
In particular, the singular values of $T^d$ are given as the $d$-fold products of singular values of $T$.
The sequence of approximation numbers $a_n\brackets{T^d}$ is hence given as the nonincreasing rearrangement of the $d$th
tensor power of the sequence $\sigma$ of singular values of $T$.

\subsection{Approximation of Mixed Order Sobolev Functions on the Torus}
\label{periodicsection}

Let $\mathbb{T}$ be the 1-torus, the circle, represented by the interval $[a,b]$,
where the two end points $a<b$ are identified.
By $L_2\brackets{\mathbb{T}}$, we denote the Hilbert space of square-integrable functions on $\mathbb{T}$,
equipped with the scalar product
\begin{equation}
 \scalar{f}{g}= \frac{1}{L} \int_{\mathbb{T}} f(x) \overline{g(x)} ~\d x
\end{equation}
and the induced norm $\norm{\cdot}$ for some $L>0$.
Typical normalizations are $[a,b]\in\set{[0,1],[-1,1],[0,2\pi]}$ and $L\in\set{1,b-a}$.
The family $\brackets{b_{k}}_{k\in \IZ}$ with
\begin{equation}
 b_{k}(x) = \sqrt{\frac{L}{b-a}} \exp\brackets{2\pi i k\,\frac{x-a}{b-a}}
\end{equation}
is an orthonormal basis of $L_2\brackets{\mathbb{T}}$, its Fourier basis, and
\begin{equation}
 \hat{f}(k) = \scalar{f}{b_{k}}
\end{equation}
is the $k$th Fourier coefficient of $f\in L_2\brackets{\mathbb{T}}$. By Parseval's identity,
\begin{equation}
 \norm{f}^2 = \sum_{k\in\IZ} |\hat{f}(k)|^2 \quad \text{and} \quad \scalar{f}{g}=\sum_{k\in\IZ} \hat{f}(k)\cdot\overline{\hat{g}(k)}.
\end{equation}
Let $w=\brackets{w_k}_{k\in\IN}$ be a nondecreasing sequence of real numbers with $w_0=1$
and let $w_{-k}=w_k$ for $k\in\IN$ and so let $\tilde w$. The univariate Sobolev space $H^{w}\brackets{\mathbb{T}}$
is the Hilbert space of functions $f\in L_2\brackets{\mathbb{T}}$ for which
\begin{equation}
 \norm{f}_w^2=\sum_{k\in\IZ} w_k^2\cdot |\hat{f}(k)|^2
\end{equation}
is finite, equipped with the scalar product
\begin{equation}
 \scalar{f}{g}_w= \sum_{k\in\IZ} w_k\hat{f}(k)\cdot\overline{w_k\hat{g}(k)}.
\end{equation}
Note that $H^{w}\brackets{\mathbb{T}}$ and $H^{\tilde{w}}\brackets{\mathbb{T}}$ coincide and their norms are equivalent,
if and only if $w \sim \tilde{w}$. In case $w_k \sim k^s$ for some $s\geq 0$, the space $H^{w}\brackets{\mathbb{T}}$ is the classical
Sobolev space of periodic univariate functions with fractional smoothness $s$, also denoted by $H^s \brackets{\mathbb{T}}$.
In particular, $H^w\brackets{\mathbb{T}}=L_2\brackets{\mathbb{T}}$ for $w\equiv 1$.

In accordance with previous notation, let $X=H^{w}\brackets{\mathbb{T}}$ and $Y=H^{\tilde w}\brackets{\mathbb{T}}$.
The embedding $T$ of $X$ into $Y$ is compact, if and only if $w_k / \tilde w_k$ tends to infinity as $k$ tends to infinity.
The Fourier basis $\brackets{b_{k}}_{k\in \IZ}$ is an orthogonal basis of $X$ consisting of eigenfunctions of $W=T^* T$
with corresponding eigenvalues
\begin{equation}
 \lambda(b_k) = \frac{\norm{b_k}_Y^2}{\norm{b_k}_X^2} = \frac{\tilde w_k^2}{w_k^2}.
\end{equation}
The $n$th approximation number $\sigma(n)$ of this embedding is the square root of the $n$th biggest eigenvalue.
Hence, replacing the Fourier weight sequences $w$ and $\tilde{w}$
by equivalent sequences does not affect the order of convergence of
the corresponding approximation numbers,
but it may drastically affect their asymptotic constants and preasymptotic behavior.
If $Y=L_2\brackets{\mathbb{T}}$, we obtain
\begin{equation}
 \sigma(n) = w_{k_n}^{-1}, \quad \text{where} \quad k_n=(-1)^n\left\lfloor n/2\right\rfloor.
\end{equation}
Note that $\sigma(1)$, the norm of the embedding $T$, is always one.

The $d$th tensor power $X^d=H^w_{\rm mix}\brackets{\mathbb{T}^d}$ of $X$ is a space of mixed order Sobolev
functions on the $d$-torus. If $w_k\sim k^s$ for some $s\geq 0$, this is the space $H^s_{\rm mix}\brackets{\mathbb{T}^d}$ of functions
with dominating mixed smoothness $s$. If even $s\in\IN_0$, this space consists of
all real-valued functions on the $d$-torus, which have a weak (or distributional) derivative of order $\alpha$
in $L_2\brackets{\mathbb{T}^d}$ for any $\alpha\in\set{0,1,\hdots,s}^d$. Of course, the same holds for the $d$th
tensor power $Y^d=H^{\tilde w}_{\rm mix}\brackets{\mathbb{T}^d}$ of $Y$.
The tensor power operator $T^d: X^d \to Y^d$ is the compact embedding of $H^w_{\rm mix}\brackets{\mathbb{T}^d}$
into $H^{\tilde w}_{\rm mix}\brackets{\mathbb{T}^d}$.
Hence, the approximation numbers of this embedding
are the nonincreasing rearrangement of the $d$th tensor power of $\sigma$.

If $\brackets{\tilde w_k/w_k}_{k\in\IN}$ is of polynomial decay, Theorem~\ref{asymptotic theorem}
and Theorem~\ref{preasymptoticstheorem} apply.
We formulate the results for the embedding of $H^s_{\rm mix}\brackets{\mathbb{T}^d}$
into $L_2\brackets{\mathbb{T}^d}$, where $H^s_{\rm mix}\brackets{\mathbb{T}^d}$ will be equipped with different equivalent norms,
indicated by the notation
\begin{equation}
 \begin{split}
  &H^{s,\circ,\gamma}_{\rm mix}\brackets{\mathbb{T}^d},	\quad \text{if} \quad w_k=\brackets{\sum_{l=0}^s \abs{\gamma^{-1} \frac{2\pi k}{b-a}}^{2l}}^{1/2},\\
  &H^{s,*,\gamma}_{\rm mix}\brackets{\mathbb{T}^d},	\quad \text{if} \quad w_k=\brackets{1+\abs{\gamma^{-1} \frac{2\pi k}{b-a}}^{2s}}^{1/2},\\
  &H^{s,+,\gamma}_{\rm mix}\brackets{\mathbb{T}^d},	\quad \text{if} \quad w_k=\brackets{1+\abs{\gamma^{-1} \frac{2\pi k}{b-a}}^{2}}^{s/2},\\
  &H^{s,\#,\gamma}_{\rm mix}\brackets{\mathbb{T}^d},	\quad \text{if} \quad w_k=\brackets{1+\abs{\gamma^{-1} \frac{2\pi k}{b-a}}}^{s},
 \end{split}
\end{equation}
for some $\gamma >0$. The last three norms are due to Kühn, Sickel and Ullrich \cite{ksu}, who study all these norms for $\gamma=1$, $L=1$ and $[a,b]=[0,2\pi]$.
The last norm is also studied by Chernov and D\~ung in \cite{cd} for $L=2\pi$, $[a,b]=[-\pi,\pi]$ and arbitrary values of $\gamma$.
If $s$ is a natural number, the first two scalar products take the form
\begin{equation}
\begin{split}
 &\scalar{f}{g}_{H^{s,\circ,\gamma}_{\rm mix}} = \sum\limits_{\alpha\in\set{0,\hdots,s}^d} \gamma^{-2s\abs{\alpha}} \scalar{\diff^\alpha f}{\diff^\alpha g},\\
 &\scalar{f}{g}_{H^{s,*,\gamma}_{\rm mix}} = \sum\limits_{\alpha\in\set{0,s}^d} \gamma^{-2s\abs{\alpha}} \scalar{\diff^\alpha f}{\diff^\alpha g}.
\end{split}
\end{equation}
This is why $H^{s,\circ,1}_{\rm mix}\brackets{\mathbb{T}^d}$ and $H^{s,*,1}_{\rm mix}\brackets{\mathbb{T}^d}$ might be considered the most natural choice.
Note that the corresponding approximation numbers of the embedding $T^d$ are independent of the normalization constant $L$,
but they do depend on the length of the interval $[a,b]$.

\begin{cor}
 The following limits exist and coincide:
 \begin{equation*}
  \left.\begin{array}{l}
        \lim\limits_{n\to\infty} a_n\brackets{H^{s,\circ,\gamma}_{\rm mix}\brackets{\mathbb{T}^d}\hookrightarrow L_2\brackets{\mathbb{T}^d}}\cdot n^s \brackets{\log n}^{-s(d-1)}\\
        \lim\limits_{n\to\infty} a_n\brackets{H^{s,*,\gamma}_{\rm mix}\brackets{\mathbb{T}^d}\hookrightarrow L_2\brackets{\mathbb{T}^d}}\cdot n^s \brackets{\log n}^{-s(d-1)}\\
	\lim\limits_{n\to\infty} a_n\brackets{H^{s,+,\gamma}_{\rm mix}\brackets{\mathbb{T}^d}\hookrightarrow L_2\brackets{\mathbb{T}^d}}\cdot n^s \brackets{\log n}^{-s(d-1)}\\
	\lim\limits_{n\to\infty} a_n\brackets{H^{s,\#,\gamma}_{\rm mix}\brackets{\mathbb{T}^d}\hookrightarrow L_2\brackets{\mathbb{T}^d}}\cdot n^s \brackets{\log n}^{-s(d-1)}
       \end{array}\right\}
  =\brackets{\frac{\brackets{\gamma\, \frac{b-a}{\pi}}^d}{\brackets{d-1}!}}^s
 .\end{equation*}
\end{cor}

Of course, this coincides with the limits computed in \cite{ksu}, if $\gamma^{-1} \frac{b-a}{\pi}=2$.
The third limit (for $[a,b]=[-\pi,\pi]$, $L=2\pi$ and any $\gamma>0$) may not be written down explicitly in \cite{cd},
but can be derived from their Theorem 4.6.

\begin{cor}
 Let $\square\in\set{\circ,*,+,\#}$. For any $s>0$, $d\in\IN$ and $n\in\set{2,\hdots,3^d}$,
 \begin{align*}
  \sigma_\square(2) \brackets{\frac{1}{n}}^{\beta_\square(d,n)} \leq
  a_n\brackets{H^{s,\square,\gamma}_{\rm mix}\brackets{\mathbb{T}^d}\hookrightarrow L_2\brackets{\mathbb{T}^d}}
  \leq\, &\brackets{\frac{\tilde{C}(\delta)}{n}}^{\alpha_\square(d,\delta)}.
 \end{align*}
 The parameter $\delta\in (0,1]$ is arbitrary, $\tilde{C}(\delta) = \exp\brackets{\brackets{3/\eta}^{1+\delta}/\delta}$
 for $\eta=\frac{2\pi}{\gamma(b-a)}$
 and the values $\sigma_\square$, $\alpha_\square$
 and $\beta_\square$ are listed below.
 The upper bound holds for all $n\in\IN$.
\end{cor}

\setlength{\extrarowheight}{5pt}
\begin{center}
\begin{tabular}[h]{c|c|c|c}
 $\square$	& $\sigma_\square(2)$	& $\alpha_\square(d,\delta)$		& $\beta_\square(d,n)$		\\
 \hline
 $\circ$	& $\brackets{\sum_{l=0}^s \eta^{2l}}^{-\frac{1}{2}}$
		& $\frac{\frac{1}{2} \log\brackets{\sum_{l=0}^s\eta^{2l}}}{\log d + \frac{1+\delta}{2s}\cdot\log\brackets{\sum_{l=0}^s\eta^{2l}}}$
		& $\frac{\frac{1}{2} \log\brackets{\sum_{l=0}^s\eta^{2l}}}{\log\brackets{1+\frac{2}{\log_3 n}\,d}}$		\\
  \hline
 $*$		& $\brackets{1+\eta^{2s}}^{-\frac{1}{2}}$
		& $\frac{\frac{1}{2} \log\brackets{1+\eta^{2s}}}{\log d + \frac{1+\delta}{2s}\cdot\log\brackets{1+\eta^{2s}}}$
		& $\frac{\frac{1}{2} \log\brackets{1+\eta^{2s}}}{\log\brackets{1+\frac{2}{\log_3 n}\,d}}$		\\
  \hline
 $+$		& $\brackets{1+\eta^{2}}^{-\frac{s}{2}}$	
		& $\frac{\frac{s}{2} \log\brackets{1+\eta^{2}}}{\log d + \frac{1+\delta}{2}\cdot\log\brackets{1+\eta^{2}}}$		
		& $\frac{\frac{s}{2} \log\brackets{1+\eta^{2}}}{\log\brackets{1+\frac{2}{\log_3 n}\,d}}$		\\
  \hline
 $\#$		& $\brackets{1+\eta}^{-s}$	
		& $\frac{s \log\brackets{1+\eta}}{\log d + (1+\delta)\log\brackets{1+\eta}}$		
		& $\frac{s \log\brackets{1+\eta}}{\log\brackets{1+\frac{2}{\log_3 n}\,d}}$
\end{tabular}
\end{center}

Let us consider the setting of \cite{ksu}, where $\gamma=1$ and $b-a=2\pi$ and hence $\eta$ is one.
The exponents $\alpha_\#(d,\delta)=\frac{s}{\log_2 d + 1+\delta}$ and
$\alpha_+(d,\delta)=\frac{s}{2\log_2 d + 1+\delta}$ in our upper bounds are slightly better
than the exponents $\frac{s}{\log_2 d + 2}$ and $\frac{s}{2\log_2 d + 4}$ in Theorem 4.9, 4.10 and Theorem 4.17 of \cite{ksu},
but almost the same. Also the lower bounds basically coincide.
Regarding $H^{s,*,1}_{\rm mix}\brackets{\mathbb{T}^d}$, Kühn, Sickel and Ullrich only studied the case $1/2\leq s\leq1$ in Theorem 4.20.
As we see now, there is a major difference between this natural norm and the last two norms:
For large dimensions $d$, the preasymptotic behavior of the approximation numbers is roughly $n^{-t_{d,\square}}$, where
\begin{equation}
 t_{d,\circ}=\frac{\log\brackets{s+1}}{2 \log d}, \quad
 t_{d,*}=\frac{1}{2 \log_2 d}, \quad
 t_{d,+}=\frac{s}{2 \log_2 d}, \quad
 t_{d,\#}=\frac{s}{\log_2 d}.
\end{equation}
This means that the smoothness of the space only has a minor or even no impact on the preasymptotic decay of the approximation numbers,
if $H^{s}_{\rm mix}\brackets{\mathbb{T}^d}$ is equipped with one of the natural norms
$\norm{\cdot}_{H^{s,\circ,1}_{\rm mix}}$ or $\norm{\cdot}_{H^{s,*,1}_{\rm mix}}$.

This changes, however, if the value of $\eta=\frac{2\pi}{\gamma (b-a)}$ changes. If $\eta$ is larger than one,
because we consider a shorter interval $[a,b]$ or because we put some weight $\gamma<\frac{2\pi}{b-a}$,
also the exponents $t_{d,\circ}$ and $t_{d,*}$ get linear in $s$.
%On the other hand, if $\eta$ is less than one, increasing the smoothness even worsens the
%preasymptotic behaviour.
For the other two families of norms, the smoothness does show and the value of $\eta$ is less important.

There are no preasymptotic estimates in \cite{cd}.

\subsection{Approximation of Mixed Order Jacobi Functions on the Cube}

The above results also apply to the approximation numbers of the embedding
of mixed order Jacobi functions on the $d$-cube in the corresponding $L_2$-space
as considered in \cite[Section 5]{cd}.

Let $\mathbb{I}$ be the 1-cube, a line segment, represented by $[-1,1]$.
For fixed parameters $\alpha,\beta>-1$ with $a:=\frac{\alpha+\beta+1}{2}>0$,
the weighted $L_2$-space $Y=L_2\brackets{\mathbb{I},w}$ is the Hilbert space of measurable,
real-valued functions on $\mathbb{I}$ with
\begin{equation}
 \int_{\mathbb{I}} f(x)^2 w(x) ~\d x < \infty,
\end{equation}
equipped with the scalar product
\begin{equation}
 \scalar{f}{g} = \int_{\mathbb{I}} f(x) g(x) w(x) ~\d x
\end{equation}
and the induced norm $\norm{\cdot}$, where $w:\mathbb{I}\to \IR$ is the Jacobi weight
\begin{equation}
 w(x)=(1-x)^\alpha (1+x)^\beta.
\end{equation}
This reduces to the classical space of square-integrable functions, if both parameters are zero.
As $\alpha$ respectively $\beta$ increases, the space grows, since we allow for
stronger singularities on the right respectively left endpoint, and vice versa.

The family of Jacobi polynomials $\brackets{P_k}_{k\in\IN_0}$ is an orthogonal basis of $Y$.
These polynomials can be defined as the unique solutions of the differential equations
\begin{equation}
 \mathcal{L} P_k= k(k+2a) P_k
\end{equation}
for the second order differential operator
\begin{equation}
 \mathcal{L} = - w(x)^{-1} \frac{d}{dx}\brackets{\brackets{1-x^2}w(x)\frac{d}{dx}}
\end{equation}
that satisfy
\begin{equation}
 P_k(1)=\binom{k+\alpha}{k}
 \quad\text{and}\quad
 P_k(-1)=(-1)^k \binom{k+\beta}{k}
.\end{equation}
We denote the $k$th Fourier coefficient of $f$ with respect to the normalized Jacobi basis by $f_k$.
The scalar product in $Y$ hence admits the representation
\begin{equation}
 \scalar{f}{g}=\sum\limits_{k=0}^\infty f_k g_k.
\end{equation}
For $s>0$ let $X=K^s\brackets{\mathbb{I},w}$ be the Hilbert space of functions $f\in Y$ with
\begin{equation}
 \sum\limits_{k=0}^\infty \brackets{1+a^{-1} k}^{2s} f_k^2 < \infty,
\end{equation}
equipped with the scalar product
\begin{equation}
 \scalar{f}{g}_s = \sum\limits_{k=0}^\infty \brackets{1+a^{-1} k}^{2s} f_k g_k
\end{equation}
and the induced norm $\norm{\cdot}_s$.
Obviously, $\brackets{P_k}_{k\in\IN_0}$ is an orthogonal basis of $X$, too.
In case $s$ is an even integer, this is the space of all functions $f\in L_2\brackets{\mathbb{I},w}$
such that $\mathcal{L}^j f\in L_2\brackets{\mathbb{I},w}$ for $j=1\hdots \frac{s}{2}$ and the scalar product
\begin{equation}
 \scalar{f}{g}_{s,*} = \sum\limits_{j=0}^{s/2} \scalar{\mathcal{L}^j f}{\mathcal{L}^j g}
\end{equation}
is equivalent to the one above.
The parameter $s$ can hence be interpreted as smoothness of the functions in $K^s\brackets{\mathbb{I},w}$.
The embedding $T$ of $X$ into $Y$ is compact
and its $n$th approximation number is given by
\begin{equation}
 \sigma(n)=a_n(T)=\frac{\norm{P_{n-1}}}{\norm{P_{n-1}}_s}=\brackets{1+a^{-1} \brackets{n-1}}^{-s}.
\end{equation}

We can apply our theorems to study the approximation numbers of the $d$th tensor power $T^d$ of $T$.
This is the embedding of $X^d=K^s\brackets{\mathbb{I}^d,w_d}$ into $Y^d=L_2\brackets{\mathbb{I}^d,w_d}$,
where $Y^d$ is the weighted $L_2$-space on the $d$-cube with respect to the
Jacobi weight $w_d=w\otimes\hdots\otimes w$ 
and $X^d$ is the subspace of Jacobi functions of mixed order $s$.
Like in the univariate case, $X^d$ can be described via differentials
of dominating mixed order $s$ and less, if $s$ is an even integer.

\begin{cor} For any $d\in\IN$ and $s>0$, the following limit exists:
\begin{equation*}
 \lim\limits_{n\to\infty} a_n\brackets{K^s\brackets{\mathbb{I}^d,w_d} \hookrightarrow L_2\brackets{\mathbb{I}^d,w_d}} \cdot n^s \brackets{\log n}^{-s(d-1)}
  =\brackets{\frac{a^d}{\brackets{d-1}!}}^s
.\end{equation*}
\end{cor}

This result could also be derived from Theorem 5.5 in \cite{cd}.
In addition, we get the following preasymptotic estimates:

\begin{cor}
 For any $\delta\in (0,1]$, $s>0$, $d\in\IN$ and $n\in\set{2,\hdots,2^d}$,
 \begin{align*}
  &\brackets{\frac{a}{a+1}}^s \brackets{\frac{1}{n}}^{p_{s,a,d,n}}
  \leq
  a_n\brackets{K^s\brackets{\mathbb{I}^d,w_d} \hookrightarrow L_2\brackets{\mathbb{I}^d,w_d}}
  \leq\,
  \brackets{\frac{\exp\brackets{\frac{(2a)^{1+\delta}}{\delta}}}{n}}^{q_{s,a,d,\delta}}\\
  &\text{with} \quad
  p_{s,a,d,n}=\frac{s \log\frac{a+1}{a}}{\log\brackets{1+ \frac{d}{\log_2 n}}} \quad\text{and}\quad
  q_{s,a,d,\delta}=\frac{s \log\frac{a+1}{a}}{\log d + (1+\delta) \log\frac{a+1}{a}}.
 \end{align*}
 The upper bound even holds for all $n\in\IN$.
\end{cor}

This means that for large dimension $d$, a preasymptotic decay of approximate order $t_d=s \log\frac{a+1}{a} /\log d$ in $n$ can be observed.

\subsection{Approximation of Mixed Order Sobolev Functions on the Cube}
\label{nonsobsection}

Another example of a tensor power operator is given by the $L_2$-embedding
of mixed order Sobolev functions on the $d$-cube.
Let $\mathbb{I}$ be the 1-cube and $\mathbb{T}$ be the 1-torus.
Both shall be represented by the interval $[a,b]$, where $a$ and $b$ are identified in the second case.
For any $s\in\IN_0$, the vector space
\begin{equation}
 H^s\brackets{\mathbb{I}}=\set{f\in L_2\brackets{\mathbb{I}} \mid f^{(l)}\in L_2\brackets{\mathbb{I}} \text{ for } 1\leq l \leq s}
,\end{equation}
equipped with the scalar product
\begin{equation}
\label{scalarproductdefinition}
 \scalar{f}{g}_s=\sum\limits_{l=0}^s \int_a^b f^{(l)}(x)\cdot \overline{g^{(l)}(x)} ~\d x
\end{equation}
and induced norm $\norm{\cdot}_s$, is a Hilbert space, the Sobolev space of order $s$ on $\mathbb{I}$.
In case $s=0$, it coincides with $L_2\brackets{\mathbb{I}}$.
The subset
\begin{equation}
 H^s\brackets{\mathbb{T}} = \set{f\in H^s\brackets{\mathbb{I}} \mid f^{(l)}(a)=f^{(l)}(b) \text{\,\, for\, } l=0,1,\hdots,s-1}
\end{equation}
of periodic functions is a closed subspace with codimension $s$,
the Sobolev space of order $s$ on $\mathbb{T}$.
By means of Parseval's identity and integration by parts, the above norm can be rearranged to
\begin{equation}
\label{normrewritten}
 \norm{f}_s^2= \sum_{k\in\IZ} \abs{\hat f (k)}^2 \sum_{l=0}^s \abs{\frac{2\pi k}{b-a}}^{2l} \quad
 \text{for } f\in H^s\brackets{\mathbb{T}}
,\end{equation}
where
\begin{equation}
 \hat f(k)=\sqrt{\frac{1}{b-a}} \int_a^b f(x)\cdot \exp\brackets{-2\pi i k \,\frac{x-a}{b-a}} \d x
\end{equation}
is the $k$th Fourier coefficient of $f$.
In the limiting case $s=\infty$, the Sobolev space $H^\infty\brackets{\mathbb{I}}$ shall be defined
as the Hilbert space
\begin{equation}
 H^\infty\brackets{\mathbb{I}}=\set{f\in \mathcal{C}^\infty\brackets{\mathbb{I}}
 \mid \sum_{l=0}^\infty \norm{f^{(l)}}_0^2 < \infty}
,\end{equation}
equipped with the scalar product (\ref{scalarproductdefinition}) for $s=\infty$.
It contains all polynomials and is hence infinite-dimensional.
The space $H^\infty\brackets{\mathbb{T}}$ shall be the closed subspace of periodic functions, i.e.
\begin{equation}
 H^\infty\brackets{\mathbb{T}} = \set{f\in H^\infty\brackets{\mathbb{I}} \mid f^{(l)}(a)=f^{(l)}(b) \text{ for any } l\in \IN_0}
.\end{equation}
Note that (\ref{normrewritten}) also holds for $s=\infty$. Hence,
\begin{equation}
 H^\infty\brackets{\mathbb{T}} = \vspan\set{\exp\brackets{2\pi i k\,\frac{\cdot -a}{b-a}} \mid k\in\IZ \text{ with } \abs{\frac{2\pi k}{b-a}} <1}
\end{equation}
is finite-dimensional with dimension $2 \lceil \frac{b-a}{2\pi}\rceil -1$.
In case $b-a\leq 2\pi$, it consists of constant functions only.

If $s$ is positive, $H^s\brackets{\mathbb{I}}$ is compactly embedded into $L_2\brackets{\mathbb{I}}$.
Let $\sigma^{(s)}(n)$ be the $n$th singular value of this embedding
and let $\tilde \sigma^{(s)}(n)$ be the $n$th singular value of the embedding of the subspace $H^s\brackets{\mathbb{T}}$ into
$L_2\brackets{\mathbb{T}}$.
We want to study the approximation numbers
of the compact embedding of the $d$th tensor power space $H^s_{\rm mix}\brackets{\mathbb{I}^d}$
into $L_2\brackets{\mathbb{I}^d}$. If $s$ is finite, this is the space
\begin{equation}
 H^s_{\rm mix}\brackets{\mathbb{I}^d} = \set{f\in L_2\brackets{\mathbb{I}^d} \mid \diff^\alpha f \in L_2\brackets{\mathbb{I}^d}
 \text{ for each } \alpha\in\set{0,\hdots,s}^d}
,\end{equation}
equipped with the scalar product
\begin{equation}
 \scalar{f}{g}_s=\sum_{\alpha\in\set{0,\hdots,s}^d}
 \int_{[a,b]^d} \diff^\alpha f(\boldsymbol x)\cdot \overline{\diff^\alpha g(\boldsymbol x)} ~\d \boldsymbol{x}
.\end{equation}
See Section~\ref{periodicsection} for a treatment of the $L_2$-approximation numbers
of the $d$th tensor power $H^s_{\rm mix}\brackets{\mathbb{T}^d}$ of the periodic space.

By means of Theorem~\ref{asymptotic theorem} and Theorem~\ref{preasymptoticstheorem}, it is enough
to study the singular values $\sigma^{(s)}$ of the embedding in the univariate case.
As we have seen in Section~\ref{periodicsection},
\begin{equation}
 \tilde\sigma^{(s)}(n)=\brackets{\sum_{l=0}^s \abs{\frac{2\pi\left\lfloor n/2\right\rfloor}{b-a}}^{2l}}^{-1/2}
 \quad\text{for } n\in \IN \text{ and } s\in\IN
\end{equation}
and in particular,
\begin{equation}
 \lim\limits_{n\to\infty} \tilde\sigma^{(s)}(n)\,n^s=\brackets{\frac{b-a}{\pi}}^s
.\end{equation}
The singular values for nonperiodic functions, on the other hand, are not known explicitly.
However, $\sigma^{(s)}$ and $\tilde \sigma^{(s)}$ interrelate as follows.

\begin{lemma}
\label{sigmanlemma}
 For any $n\in\IN$ and $s\in\IN$, it holds that $\sigma^{(s)}(n+s)\leq \tilde\sigma^{(s)}(n)\leq \sigma^{(s)}(n)$.
\end{lemma}

\begin{proof}
 The second inequality is obvious, since $H^s\brackets{\mathbb{T}}$ is a subspace of $H^s\brackets{\mathbb{I}}$.
 %It holds true for $s=\infty$.
 %On the other hand, the orthogonal complement of $H^s\brackets{\mathbb{T}}$
 %is the span of the first $s$ Bernoulli polynomials,
 %if $H^s\brackets{\mathbb{I}}$ is equipped with a different but equivalent scalar product, see \cite[Section~10.2]{wahba}.
 %In particular, it is an $s$-dimensional algebraic complement and this property is independent of the scalar product.
 %In particular, the codimension of $H^s\brackets{\mathbb{T}}$ in $H^s\brackets{\mathbb{I}}$ is $s$.
 The first inequality is true, since the codimension of this subspace is $s$.
 Let $U$ be the orthogonal complement of of $H^s\brackets{\mathbb{T}}$ in $H^s\brackets{\mathbb{I}}$.
 By relation~(\ref{minmax}),
 \begin{equation}
 \begin{split}
  \sigma^{(s)}(n+s)\,=
  &\min\limits_{\substack{V\subseteq H^s\brackets{\mathbb{I}}\\ \dim(V)\leq n+s-1}}\, \max\limits_{\substack{f\in H^s\brackets{\mathbb{I}}, f\perp V\\ \norm{f}_s=1}} \norm{f}_0
  \,\leq \min\limits_{\substack{\tilde V\subseteq H^s\brackets{\mathbb{T}}\\ \dim(\tilde V)\leq n-1}}\, \max\limits_{\substack{f\in H^s\brackets{\mathbb{I}}, \norm{f}_s=1\\f\perp(\tilde V\oplus U)}} \norm{f}_0\\
  &= \min\limits_{\substack{\tilde V\subseteq H^s\brackets{\mathbb{T}}\\ \dim(\tilde V)\leq n-1}}\, \max\limits_{\substack{f\in H^s\brackets{\mathbb{T}}, f \perp \tilde V\\\norm{f}_s=1}} \norm{f}_0
  \,=\, \tilde\sigma^{(s)}(n)
 .\end{split}
 \end{equation}
  Note that the same argument is not valid for $d>1$. In this case, the codimension of
  $H^s_{\rm mix}\brackets{\mathbb{T}^d}$ in $H^s_{\rm mix}\brackets{\mathbb{I}^d}$ is not finite.
\end{proof}

Lemma~\ref{sigmanlemma} implies that the asymptotic constants of the approximation numbers for the periodic and the nonperiodic
functions coincide in the univariate case:
\begin{equation}
\begin{split}
 \lim\limits_{n\to\infty} n^s\tilde\sigma^{(s)}(n)
 &\leq \lim\limits_{n\to\infty} n^s\sigma^{(s)}(n) 
 = \lim\limits_{n\to\infty} (n+s)^s\sigma^{(s)}(n+s)\\
 &= \lim\limits_{n\to\infty} n^s\sigma^{(s)}(n+s) 
 \leq \lim\limits_{n\to\infty} n^s\tilde\sigma^{(s)}(n).
\end{split}
\end{equation}
Theorem~\ref{asymptotic theorem} implies that they also coincide in the multivariate case.

\begin{cor} For any $d\in\IN$ and $s\in\IN$, the following limit exists:
\begin{equation*}
 \lim\limits_{n\to\infty} a_n\brackets{H^s_{\rm mix}\brackets{\mathbb{I}^d} \hookrightarrow L_2\brackets{\mathbb{I}^d}} \cdot n^s \brackets{\log n}^{-s(d-1)}
  =\brackets{\frac{\brackets{b-a}^d}{\pi^d \brackets{d-1}!}}^s
.\end{equation*}
\end{cor}

As depicted in Section~\ref{preasymptoticssection}, the approximation numbers show a preasymptotic decay
of approximate order $\frac{\log \sigma^{(s)}(2)^{-1}}{\log d}$.
Lemma~\ref{sigmanlemma} gives no information on $\sigma^{(s)}(2)$. % and is hence not suited for preasymptotic estimates.
However, relation~(\ref{minmax}) implies that
\begin{equation}
 \sigma^{(\infty)}(2) = \max\limits_{ f\perp 1,\, f\neq 0} \frac{\norm{f}_0}{\norm{f}_\infty}
 \geq \frac{\norm{2x-a-b}_0}{\norm{2x-a-b}_\infty}
 = \sqrt{\frac{(b-a)^2}{12+(b-a)^2}}.
\end{equation}
If, for example, the length of the interval $\mathbb{I}$ is one, we obtain
\begin{equation}
 \sigma^{(\infty)}(2) \geq 0.27735
.\end{equation}
Since any lower bound on the approximation numbers for $s=\infty$ is a lower bound for $s\in\IN$,
Theorem~\ref{preasymptoticstheorem} yields the following corollary.

\begin{cor}
 For any $d\in\IN$, any $s\in\IN\cup\set{\infty}$ and $d<n\leq 2^d$,
 \begin{align*}
  &a_n\brackets{H^s_{\rm mix}\brackets{[0,1]^d} \hookrightarrow L_2\brackets{[0,1]^d}}
   \geq\, 0.27 \cdot n^{-c(d,n)},\\
 &\text{where}\quad\quad c(d,n) = \frac{1.2825}{\log \brackets{1+\frac{2 d}{\log_2 n}}}\leq 1.17
.\end{align*}
\end{cor}

On the other hand, any upper bound on the approximation numbers for $s=1$ is an upper bound for $s\geq 1$.
The singular values $\sigma^{(s)}(n)$ for $s=1$ are known.
Let $T_s$ be the compact embedding of $H^s\brackets{\mathbb{I}}$ into $L_2\brackets{\mathbb{I}}$ and let $W_s=T_s^*T_s$.
Then $\sigma^{(s)}(n)$ is the square-root of the $n$th largest eigenvalue of $W_s$.
It is shown in \cite{agnan} that the family $\brackets{b_k}_{k\in\IN_0}$ is a complete orthogonal system in $H^1\brackets{\mathbb{I}}$,
where the function $b_k:\mathbb{I}\to\IR$ with
\begin{equation}
 b_k(x)=\cos\brackets{k\pi\cdot\frac{x-a}{b-a}}  \quad\text{for } k\in\IN_0
\end{equation}
is an eigenfunction of $W_1$ with respective eigenvalue
\begin{equation}
 \lambda_k=\brackets{1+\brackets{\frac{k\pi}{b-a}}^2}^{-1}.
\end{equation}
In case $\mathbb{I}=[0,1]$,
\begin{equation}
 \sigma^{(1)}(2)=\brackets{\sqrt{1+\pi^2}}^{-1}\leq 0.30332
\end{equation}
and
\begin{equation}
 \sigma^{(1)}(n)\leq 0.607 \cdot n^{-1}
\end{equation}
for $n\geq 2$. Theorem~\ref{preasymptoticstheorem} for $\delta=0.65$ yields the following upper bound.

\begin{cor}
 For any $d\in\IN$, any $s\in\IN\cup\set{\infty}$ and $n\in\IN$,
 \begin{equation*}
  a_n\brackets{H^s_{\rm mix}\brackets{[0,1]^d} \hookrightarrow L_2\brackets{[0,1]^d}}
  \leq\, \brackets{\frac{2}{n}}^{c(d)}
  \quad\text{with}\quad 
  c(d)=\frac{1.1929}{2+\log d}.
  \end{equation*}
\end{cor}
% sigman leq 0.607 n^-1 for n geq 2

Apparently, the upper bound for $s=1$ and the lower bound for $s=\infty$ are already close.
The gap between the cases $s=2$ and $s=\infty$ is even smaller.

Let $c$ be the midpoint of $\mathbb{I}$ and let $l$ be its radius.
Moreover, let $\hat\omega = \sqrt{1+\omega^2}$ for $\omega\in\IR$ and consider the countable sets
\begin{equation}
 \begin{split}
  &I_1=\set{\omega\geq 0\mid \hat\omega^3\cosh(\hat\omega l)\sin(\omega l)
  + \omega^3\sinh(\hat\omega l)\cos(\omega l)=0},\\
  &I_2=\set{\omega> 0\mid \hat\omega^3\sinh(\hat\omega l)\cos(\omega l)
  - \omega^3\cosh(\hat\omega l)\sin(\omega l)=0}.
 \end{split}
\end{equation}
It can be shown (with some effort) that the family $\brackets{b_\omega}_{\omega\in I_1\cup I_2}$
is a complete orthogonal system in $H^2\brackets{\mathbb{I}}$,
where the function $b_\omega:\mathbb{I}\to\IR$ with
\begin{equation}\begin{split}
  &b_\omega(x)
  =\omega^2\cdot \frac{\cosh\brackets{\hat \omega (x-c)}}{\cosh\brackets{\hat\omega l}}
  +\hat\omega^2\cdot \frac{\cos\brackets{\omega (x-c)}}{\cos\brackets{\omega l}},\quad \text{if }\omega\in I_1,\\
  &b_\omega(x)
  =\omega^2\cdot \frac{\sinh\brackets{\hat \omega (x-c)}}{\sinh\brackets{\hat\omega l}}
  +\hat\omega^2\cdot \frac{\sin\brackets{\omega (x-c)}}{\sin\brackets{\omega l}},\quad \text{if }\omega\in I_2,
\end{split}
\end{equation}
is an eigenfunction of $W_2$ with respective eigenvalue
\begin{equation}
 \lambda_\omega=\brackets{1+\omega^2+\omega^4}^{-1}
.\end{equation}
In particular,
\begin{equation}
 \sigma^{(2)}(2) = \brackets{\sqrt{1+\omega_0^2+\omega_0^4}}^{-1},
\end{equation}
where $\omega_0$ is the smallest nonzero element of $I_1\cup I_2$.
If, for example, the interval $\mathbb{I}$ has unit length, we obtain
\begin{equation}
 \sigma^{(2)}(2)\leq 0.27795 %946etc
\end{equation}
and like before,
\begin{equation}
 \sigma^{(2)}(n)\leq 0.607 \cdot n^{-1}
\end{equation}
for $n\geq 2$.
Theorem~\ref{preasymptoticstheorem} for $\delta=0.65$ yields the following upper bound.

\begin{cor}
 For any $d\in\IN$, any $s\in\IN\cup\set{\infty}$ with $s\geq 2$ and $n\in\IN$,
 \begin{equation*}
  a_n\brackets{H^s_{\rm mix}\brackets{[0,1]^d} \hookrightarrow L_2\brackets{[0,1]^d}}
  \leq\, \brackets{\frac{2}{n}}^{c(d)}
  \quad\text{with}\quad 
  c(d)=\frac{1.2803}{2+\log d}.
  \end{equation*}
\end{cor}

In short, the preasymptotic rate of the $L_2$-approximation numbers
of mixed order $s$ Sobolev functions on the unit cube is $\frac{1.1929}{\log d}$ for $s=1$,
and in between $\frac{1.2803}{\log d}$ and $\frac{1.2825}{\log d}$ for any other $s\in\IN\cup\set{\infty}$.

\section{Tractability through Decreasing Complexity of the Univariate Problem}
\label{tracsection}

For every $d\in\IN$, let $X_d$ and $Y_d$ be normed spaces and let $F_d$ be a subset of $X_d$.
We want to approximate the operator $T_d: F_d \to Y_d$ by an algorithm $A_n: F_d \to Y_d$ that uses at most $n$
linear and continuous functionals on $X_d$.
The $n$th minimal worst case error
\begin{equation}
 e(n,d)=\inf\limits_{A_n} \sup\limits_{f\in F_d} \norm{T_d f- A_n f}_{Y_d}
\end{equation}
measures the worst case error of the best such algorithm $A_n$.
If $F_d$ is the unit ball of a pre-Hilbert space and $T_d$ is linear,
it is known to coincide with the $(n+1)$th approximation number of $T_d$.
Conversely, the information complexity
\begin{equation}
 n(\varepsilon,d)=\min \set{n\in\IN_0 \mid e(n,d) < \varepsilon}
\end{equation}
is the minimal number of linear and continuous functionals
that is needed to achieve an error less than $\varepsilon$.
The problem $\set{T_d}$ is called polynomially tractable, if there are nonnegative numbers $C$, $p$ and $q$ such that
\begin{equation}
\label{trackdef}
 n(\varepsilon,d) \leq C\, \varepsilon^{-q}\, d^p  \quad\quad \text{for all } d\in\IN \text{ and } \varepsilon>0.
\end{equation}
It is called strongly polynomially tractable, if (\ref{trackdef}) holds with $p$ equal to zero.
See \cite{track1} for a detailed treatment of these and other concepts of tractability.

In the following, $X_d$ and $Y_d$ will be Hilbert spaces and $T_d$ will be a linear and compact norm-one operator
with approximation numbers of polynomial decay.
For example, one can think of $T_d$ as the embedding of the Sobolev space $H^{s_d}(G)$ into $H^{r_d}(G)$
for some $r_d<s_d$ and a compact manifold $G$.
Let $T_d^d$ be the $d$th tensor power of $T_d$.
In the chosen example, this is the embedding of $H^{s_d}_{\rm mix}\brackets{G^d}$ into $H^{r_d}_{\rm mix}\brackets{G^d}$.
We will refer to $\set{T_d}$ as the univariate and to $\set{T_d^d}$ as the multivariate problem.
It is proven in \cite[Theorem~5.5]{track1} that the multivariate problem is not polynomially tractable,
if $T_d$ is the same operator for every $d\in\IN$.
This corresponds to the case, where the complexity of the univariate problem is constant in $d$.
Can we achieve polynomial tractability of the multivariate problem,
if the complexity of the univariate problem decreases, as $d$ increases?
If yes, to which extent do we have to simplify the univariate problem?
The answer is given by the following theorem.

\begin{thm}
\label{tractabilitytheorem}
 For every natural number $d$, let $T_d$ be a compact norm-one operator between Hilbert spaces
 and let $T_d^d$ be its $d$th tensor power.
 Assume that $a_n\brackets{T_d}$ is nonincreasing in $d$
 and $a_n\brackets{T_1}$ decays polynomially in $n$.
 The problem $\set{T_d^d}$ is strongly polynomially tractable,
 iff it is polynomially tractable, iff $a_2\brackets{T_d}$ decays polynomially in $d$.
\end{thm}

\begin{proof}
 Clearly, strong polynomial tractability implies polynomial tractability.
 
 Let $\set{T_d^d}$ be polynomially tractable and choose
 nonnegative numbers $C,p$ and $q$ such that
 \begin{equation}
  n(\varepsilon,d)=\#\set{n\in\IN \mid a_n(T_d^d) \geq \varepsilon} \leq C\, \varepsilon^{-q}\, d^p
 \end{equation}
 for all $\varepsilon>0$ and $d\in\IN$. In particular, there is an $r\in\IN$ with
 \begin{equation}
  n\brackets{d^{-1},d} \leq d^{r} -1
 \end{equation}
 for every $d\geq 2$.
 If $d$ is large enough, we can apply Part~$(ii)$ of Theorem~\ref{preasymptoticstheorem} for $n=d^r$ and the estimate
 \begin{equation}
  \beta\brackets{d,d^r} = \frac{\log a_2(T_d)^{-1}}{\log \brackets{1+\frac{v\cdot d}{r \log_{1+v} d}}}
  \leq \frac{2 \log a_2(T_d)^{-1}}{\log d}
 \end{equation}
 to obtain
 \begin{equation}
  d^{-1}
  > a_{d^r}(T_d^d)
  \geq a_2(T_d) \cdot d^{-r \beta\brackets{d,d^r}}
  \geq a_2(T_d)^{2r +1}.
 \end{equation}
 Consequently, $a_2(T_d)$ decays polynomially in $d$.
 
 Now let $a_2(T_d)$ be of polynomial decay.
 Then there are constants $p>0$ and $d_0\in\IN$ such that $a_2(T_d)$ is bounded above by $d^{-p}$ for any $d\geq d_0$.
 On the other hand, there are positive constants $C$ and $s$ such that
 \begin{equation}
  a_n(T_d) \leq a_n(T_1) \leq C \, n^{-s}.
 \end{equation}
 We apply Part~$(i)$ of Theorem~\ref{preasymptoticstheorem} and the estimate
 \begin{equation}
  \alpha\brackets{d,1}= \frac{\log a_2(T_d)^{-1}}{\log d + \frac{2}{s} \log a_2(T_d)^{-1}}
  \geq \frac{p}{1+\frac{2p}{s}} = r >0
 \end{equation}
 to obtain
 \begin{equation}
  a_n(T_d^d) \leq \brackets{\frac{\exp\brackets{C^{2/s}}}{n}}^r
 \end{equation}
 for any $n\in\IN$ and $d\geq d_0$. Consequently,
 \begin{equation}
  n(\varepsilon,d)=\#\set{n\in\IN \mid a_n(T_d^d) \geq \varepsilon}
  \leq \exp\brackets{C^{2/s}}\cdot \varepsilon^{-1/r}
 \end{equation}
 for any $d\geq d_0$ and $\varepsilon>0$ and $\set{T_d^d}$ is strongly polynomially tractable.
\end{proof}

Let us consider the spaces $H^{s}_{\rm mix}\brackets{\mathbb{I}^d}$ and $H^{s}_{\rm mix}\brackets{\mathbb{T}^d}$
as defined in Section~\ref{nonsobsection}.
The $L_2$-approximation in these spaces is not polynomially tractable.
Can we achieve polynomial tractability by increasing the smoothness with the dimension?

\begin{cor}
\label{tractabilitycorollary}
 The problem $\set{H^{s_d}_{\rm mix}\brackets{\mathbb{I}^d}\hookrightarrow L_2\brackets{\mathbb{I}^d}}$ is not polynomially tractable
 for any choice of natural numbers $s_d$.
 The problem $\set{H^{s_d}_{\rm mix}\brackets{\mathbb{T}^d}\hookrightarrow L_2\brackets{\mathbb{T}^d}}$ is strongly polynomially tractable,
 iff it is polynomially tractable,
 iff $b-a<2\pi$ and $s_d$ grows at least logarithmically in $d$
 or $b-a=2\pi$ and $s_d$ grows at least polynomially in $d$.
\end{cor}

With regard to tractability, the $L_2$-approximation of mixed order Sobolev functions
is hence much harder for nonperiodic than for periodic functions.
The negative tractability result for nonperiodic functions can be explained by the difficulty of approximating $d$-variate
polynomials with degree one or less in each variable and $H^1_{\rm mix}$-norm less than one.
The corresponding set of functions is contained in the unit ball
of the nonperiodic space $H^s_{\rm mix}$ for every $s\in\IN\cup\set{\infty}$.

Note that Corollary~\ref{tractabilitycorollary} for cubes of unit length is in accordance with \cite{pw},
where Papageorgiou and Woźniakowski prove the corresponding statement for the $L_2$-approximation in Sobolev spaces
of mixed smoothness $(s_1,\hdots,s_d)$ on the unit cube.
The smoothness of such functions increases from variable to variable,
but the smoothness with respect to a fixed variable does not increase with the dimension.
%This makes positive tractability results stronger and negative tractability results weaker.
There, the authors raise the question for a characterization of spaces and their norms
for which increasing smoothness yields polynomial tractability.
Theorem~\ref{tractabilitytheorem} says that in the setting of uniformly increasing mixed smoothness,
polynomial tractability is achieved,
if and only if it leads to a polynomial decay of the second singular value of the univariate problem.
It would be interesting to verify whether the same holds in the case of variable-wise increasing smoothness
and to compute the exponents of strong polynomial tractability.

The reason for the great sensibility of the tractability results for the periodic spaces to the length of the interval
can be seen in the difficulty of approximating
trigonometric polynomials with frequencies in $\frac{2\pi}{b-a}\set{-1,0,1}^d$
that are contained in the unit ball of $H^\infty_{\rm mix}\brackets{\mathbb{T}^d}$.
The corresponding set of functions is nontrivial,
if and only if $\frac{2\pi}{b-a}$ is smaller than one.

It may yet seem unnatural that the approximation numbers are so
sensible to the representation $[\boldsymbol a,\boldsymbol b]$ of the $d$-torus or the $d$-cube.
This can only happen, since the above and common scalar products
\begin{equation}
 \scalar{f}{g}=\sum_{\alpha\in\set{0,\hdots,s}^d} \scalar{\diff^\alpha f}{\diff^\alpha g}_{L_2}
\end{equation}
do not define a homogeneous family of norms on $H^{s}_{\rm mix}\brackets{[\boldsymbol a,\boldsymbol b]}$.
To see that, let $T$ be the embedding of $H^{s}_{\rm mix}\brackets{[\boldsymbol a,\boldsymbol b]}$
into $L_2\brackets{[\boldsymbol a,\boldsymbol b]}$
and let $T_0$ be the embedding in the case $[\boldsymbol a,\boldsymbol b]=[0,1]^d$.
The dilation operation $Mf=f\brackets{\boldsymbol a+(\boldsymbol b-\boldsymbol a) \,\cdot}$ defines a
linear homeomorphism both
from $L_2\brackets{[\boldsymbol a,\boldsymbol b]}$ into $L_2\brackets{[0,1]^d}$
and from $H^{s}_{\rm mix}\brackets{[\boldsymbol a,\boldsymbol b]}$ into $H^{s}_{\rm mix}\brackets{[0,1]^d}$
and 
\begin{equation}
 T_0 = M T M^{-1}
.\end{equation}
The $L_2$-spaces satisfy the homogeneity relation
\begin{equation}
 \norm{Mf}_{L_2\brackets{[0,1]^d}}
 = \lambda^d\brackets{[\boldsymbol a,\boldsymbol b]} \cdot \norm{f}_{L_2\brackets{[\boldsymbol a,\boldsymbol b]}}
 \quad \text{for} \quad f\in L_2\brackets{[\boldsymbol a,\boldsymbol b]}
.\end{equation}
If the chosen family of norms on $H^{s}_{\rm mix}\brackets{\mathbb{T}^d}$ is also homogeneous, i.e.
\begin{equation}
 \norm{Mf}_{H^{s}_{\rm mix}\brackets{[0,1]^d}}
 = \lambda^d\brackets{[\boldsymbol a,\boldsymbol b]} \cdot \norm{f}_{H^{s}_{\rm mix}\brackets{[\boldsymbol a,\boldsymbol b]}}
 \quad \text{for} \quad f\in H^{s}_{\rm mix}\brackets{[\boldsymbol a,\boldsymbol b]}
,\end{equation}
the approximation numbers of $T$ and $T_0$ clearly must coincide.
The above scalar products do not yield a homogeneous family of norms.
An example of an equivalent and homogeneous family of norms on $H^{s}_{\rm mix}\brackets{[\boldsymbol a,\boldsymbol b]}$
is defined by the scalar products
\begin{equation}
 \scalar{f}{g}=\sum_{\alpha\in\set{0,\hdots,s}^d} (\boldsymbol b - \boldsymbol a)^{2\alpha} \scalar{\diff^\alpha f}{\diff^\alpha g}_{L_2}
.\end{equation}
% im arXiv Vorzeichenfehler !
Hence, the approximation numbers and tractability results with respect to this scalar product
do not depend on $\boldsymbol a$ and $\boldsymbol b$ at all.
They coincide with the approximation numbers with respect to the previous
scalar product on $H^{s}_{\rm mix}\brackets{[0,1]^d}$.

% The reason for this phenomenon can be seen in the difficulty of $L_2$-approximation in the set
%of $(b-a)$-periodic oscillations $c\cdot \exp\brackets{i \eta \scalar{\boldsymbol k}{\boldsymbol x}}$ with $\eta=\frac{2\pi}{b-a}$, frequencies
%$\boldsymbol k \in\set{-1,0,1}^d$ and coefficients $\abs{c}<c_0$.
%There is a $c_0>0$ such that this set is contained in the unit ball of
%$H^{s}_{\rm mix}\brackets{\mathbb{T}^d}$ for every $s\in\IN$,
%if and only if $\eta$ is smaller than one.

\raggedright

\end{document}